\documentclass[12pt,reqno]{amsart}
\usepackage[notref,notcite]{}
\usepackage{amssymb}
\usepackage{amsmath}

\def\Deltav{\stackrel{ v} \Delta{}\!}

\numberwithin{equation}{section}

\newdimen\minuskern\minuskern=1pt
\def\_{\kern\minuskern{-}\kern\minuskern}
\def\+{\kern\minuskern{+}\kern\minuskern}
\def\(#1\){{\XIIpoint\rm(#1)}}

\newcommand\Deltavv{{\stackrel{v}{\kern.3mm\Delta}}}
\newcommand\nablav{{\stackrel{v}{\kern.3mm\nabla}}}
\newcommand\nablah{{\stackrel{h}{\kern.3mm\nabla}}}
\newcommand\const{{\mathrm{const}}}
\newcommand\tz{\kern.3mm{\rm{;}}\ }
\newcommand\dv{\kern.3mm{\rm{:}}\ }
\renewcommand\({$($}
\renewcommand\){$)$}
\newcommand\vop{\kern.3mm{\rm{?}}\ }

\renewcommand\emptyset\varnothing
\newcommand\Hom{\operatorname{Hom}}

\newcommand\Ker{\operatorname{Ker}}
\newcommand\Ran{\operatorname{Ran}\,}
\newcommand\supp{\operatorname{supp}}

\allowdisplaybreaks

\begin{document}

\title[Conformal Killing Tensor Fields]{On Conformal Killing Symmetric Tensor Fields\\
on Riemannian Manifolds}
\author{N.~S.~Dairbekov and V.~A.~Sharafutdinov}
\address{Kazakh-British Technical University,
Almaty, 050000 Kazakhstan}
\email{Nurlan.Dairbekov@gmail.com}
\address{Sobolev Institute of Mathematics,
     Novosibirsk, 630090 Russia}
\email{sharaf@math.nsc.ru}

\setcounter{section}{0}
\setcounter{page}{1}
\newtheorem{Theorem}{Theorem}[section]
\newtheorem{Lemma}[Theorem]{Lemma}
\newtheorem{Corollary}[Theorem]{Corollary}
\newtheorem{proposition}{Proposition}[section]
\numberwithin{equation}{section}

\begin{abstract}
 A vector field on a Riemannian manifold is called
   {\it conformal Killing}
 if it generates one-parameter group of conformal transformation. The class of
 conformal Killing
 symmetric tensor fields of an arbitrary
 rank is a natural generalization of the class of conformal Killing vector fields, and appears in different
 geometric and physical problems. We prove the statement: A trace-free
 conformal Killing
 tensor field is identically zero if
  it vanishes on some hypersurface. This statement is a basis of
  the theorem on decomposition of a symmetric tensor field
  on a compact manifold with boundary to a sum of three fields
  of special types. We also  establish triviality of the space of
  trace-free
  conformal Killing
  tensor fields on some closed manifolds.
\end{abstract}

\maketitle

\section{Introduction}

 {\it Conformal transformation}
 of a Riemannian manifold $(M,g)$ is a diffeomorphism
 $\varphi:M\rightarrow M$ such that $\varphi^*g=\lambda g$ for some
 positive function~$\lambda$ on~$M$.
   A vector field~$u$ on~$M$ is called
   {\it conformal Killing}
   if it generates one parameter transformation group
  of conformal mappings. In local coordinates, a
  conformal Killing
  vector field satisfies the equation
\begin{equation}
\frac {1} {2}({\nabla}_{\!i}u_j+{\nabla}_{\!j}u_i)=
vg_{ij}
                                                \label{(1.1)}
\end{equation}
for some scalar function~$v$ (depending on~$u$).
 Here $\nabla_i u_j$ denote the covariant derivatives of the field~$u$.

The notion of conformal Killing
 tensor fields is a generalization of the notion of conformal Killing vector fields to the case of
 higher rank tensors, and the equation that defines the first class of fields
 generalizes equation~\eqref{(1.1)}.

Given a Riemannian manifold~$(M,g)$, let
 $C^\infty(S^m\tau'_M)$ be
 the space of smooth symmetric covariant tensor field
 of rank~$m$
 on~$M$. The first order differential operator
$$
d=\sigma{\nabla}:C^\infty(S^{m-1}\tau'_M)\rightarrow
C^\infty(S^{m}\tau'_M),
$$
is called the
 {\it inner derivative.}
 Here $\nabla$ denotes the covariant derivative, and $\sigma$~is
 the symmetrization.
 The {\it divergence}
$$
\delta:C^\infty(S^{m}\tau'_M)\rightarrow C^\infty(S^{m-1}\tau'_M)
$$
is defined in local coordinates by
$ (\delta u)_{i_1\dots i_{m-1}}=
g^{jk}{\nabla}_{\!j}u_{ki_1\dots i_{m-1}} $.
The operators~$d$ and~$-\delta$ are dual to each other
 with respect to the natural $L^2$-product on the space of symmetric
 tensor fields (see~Section\,3).
 We denote by
$$
i:C^\infty(S^{m}\tau'_M)\rightarrow
C^\infty(S^{m+2}\tau'_M)
$$
the following algebraic operator of symmetric multiplication by the
 metric tensor:
$ iu=\sigma(g\otimes u) $, and the adjoint of~$i$ is denoted by
$ j:C^\infty(S^{m+2}\tau'_M)\rightarrow C^\infty(S^{m}\tau'_M) $. In local coordinates,
$ (ju)_{i_1\dots i_m}=g^{jk}u_{jki_1\dots i_m} $. The tensor field $ ju $ is called the
 {\it trace}
 of the field $u$.
 A symmetric tensor field $u$ is called
 {\it trace-free}
 if its trace is identically equal to zero:
\begin{equation}
ju=0.
                                    \label{(1.2)}
\end{equation}

A {\it Killing} tensor field is a symmetric tensor field $u$ satisfying $du=0$.
A {\it conformal Killing} tensor field is a symmetric tensor field $u$ satisfying the equation
\begin{equation}
du=iv
                                   \label{(1.3)}
\end{equation}
for some $v$. Equation~\eqref{(1.3)} is a natural
 generalization of~\eqref{(1.1)}.

Conformal Killing
 vector (covector) fields are the classic object of
 the Riemannian geometry.
 Conformal Killing symmetric tensor fields
 of higher rank
 naturally appear in various problems of physics and geometry
 (see~\cite{[Wr],[Ed], [Ge], [Mi], [Va], [Ea], [Ca]}).  There are also many papers
  where antisymmetric
  conformal Killing
  tensor fields are studied
 (conformal Killing forms) because of their role in
  Gravitation Theory and in the Maxwell equations
 (see~\cite{[St],[Jz]} and the references there).

 It should be observed that (in some sense) there are
 ``too many''
 conformal Killing
 fields of rank $m\ge 2$.
 Indeed, if $v$ is any field of rank $m-2$
  then the field $u=iv$ satisfies
 the equation $du=i(dv)$ because the operators~$i$ and~$d$ commute.
 For this reason, it makes sense to study the
 {\it trace-free conformal Killing fields}, i.e.,
 conformal Killing fields~$u$
 satisfying  equation~\eqref{(1.2)}. In some papers (for example, see~\cite{[Ea], [Ca]}),
 equation~\eqref{(1.2)} is included into the definition of a
 conformal Killing tensor field,
 but we prefer to speak on the
 {\it trace-free conformal Killing fields} in this case.

Eliminating  $v$ from equation~\eqref{(1.3)}, we get $pdu=0$,
   where $ p $ is an algebraic operator defined in Section 3 below.
   As is shown in Theorem 5.1 below, the operator $\delta pd$ is elliptic
   on the bundle of trace-free tensor fields.
   Therefore, the equation $\delta pd u=0$  implies the following statement:
   {\it a trace-free conformal Killing field is smooth}.
   Hereafter in the paper, the term ``smooth'' always means ``$C^\infty$-smooth''.

The definition of a (trace-free)
  conformal Killing field
  is invariant with respect to a conformal change of the metric in the
  following sense. If $u\in C^\infty(S^m\tau'_M)$ is a (trace-free)
  conformal Killing field
  with respect to a Riemannian metric~$g$
  then~$\lambda^m u$ is a (trace-free)
  conformal Killing field
  with respect to the metric $\lambda g$
  for any smooth positive function~$\lambda$ on~$M$.

As well known, the space of
  conformal Killing vector fields
  on~$M$ has a finite dimension if $n=\dim M\ge 3$.
  In the two-dimensional case, the space can be of infinite dimension.
  Nevertheless, in every dimension,
  a conformal Killing vector field
  is uniquely determined by its $C^\infty$-jet at any point.
  The following theorem is a generalization of the latter fact.

\begin{Theorem}\label{1.1}
Let $(M,g)$~be a connected $ n $-dimensional Riemannian manifold$,$ with $n\ge 3$.
 If a~trace-free conformal Killing
 symmetric field~$u$ of rank $m\ge 0$ satisfies the conditions
\begin{equation}
u(x_0)=0,\quad{\nabla}u(x_0)=0,\quad\dots,\quad{\nabla}^lu(x_0)=0
                                         \label{(1.4)}
\end{equation}
at some point $x_0\in M$, where $l=l(m)\leq 6m$ depends only on~$m$,
  then $u\equiv 0$. In particular, the dimension of the space of trace-free
  conformal Killing fields of rank~$m$ is finite.
  If $n=2$ then the first statement is true if~\eqref{(1.4)}
  is replaced by the following condition:
  All the derivatives of the field~$u$ vanish at the point~$x_0$.
\end{Theorem}

The theorem was first proved in~\cite{[mp]}, with
  $l(m)=\nobreak6m$. The results of~\cite{[Ea],[Ca]} provide
  the exact value $l(m)=2m$ (in the case of $m=2$, see also~\cite{[Wr]}).
  Arguments of~\cite{[Ea],[Ca]} are based on delicate facts of representation theory and the theory
  of overdetermined systems.
 At the same time, the proof given in \cite{[mp]} is quite elementary although it does not allow to obtain the exact
 value of~$l(m)$.
 For reader's convenience, we reproduce the latter proof in Sections 9 and~11.
For $n\ge3$, the scheme of the proof is as follows.
Being written in local coordinates, \eqref{(1.2)} and \eqref{(1.3)}
   constitute a linear homogeneous system of equations in components
   of the~fields $u$ and $v$, and their first order partial derivatives.
 We differentiate these equations~$l$ times and show that
 the resulting system can be solved
 with respect to all partial derivatives
 of the highest order~$l+1$. This means that the components
 of the tensors $u(t)=u\big(x(t)\big)$ and $v(t)=v\big(x(t)\big)$
 satisfy a homogeneous linear system of ordinary differential
 equations of order~$l+1$ along every smooth curve $x=x(t)$.
 Together with homogeneous initial conditions~\eqref{(1.4)},
 this yields the required result.
 In the two-dimensional case, the proof consists of reducing
 system \eqref{(1.2)}--\eqref{(1.3)} to the Cauchy-Riemann equations.

\begin{Corollary}\label{1.2}
Let $(M,g)$~be a connected Riemannian manifold of dimension $n\ge 3$.
 If tensor fields $u\in C^\infty(S^m\tau'_M)$ and
 $v\in C^\infty(S^{m-1}\tau'_M)$ $(m\ge 0)$ satisfy equation~\eqref{(1.3)} and
 initial conditions~\eqref{(1.4)} with the same $l=l(m)$ as in Theorem~$\ref{1.1}$
 then there exists a field $w\in C^\infty(S^{m-2}\tau'_M)$ such that
 $u=iw$ and $v=dw$. For $n=2$, the statement is true if~\eqref{(1.4)} is replaced by the
 following condition: All the derivatives of the field~$u$ vanish at the point~$x_0$.
\end{Corollary}

 Theorem \ref{1.1} implies Corollary~\ref{1.2} by an
   algebraic trick presented in~Section~3. Theorem \ref{1.1} is also used
   in the proof of the following proposition.

\begin{Theorem}\label{1.3}
Let $(M,g)$~ be a connected Riemannian manifold of dimension at least~$2$,
 and let $\emptyset\,{\neq}\,\Gamma\,{\subset}\, M$ be a smooth
 hypersurface.
 In particular$,$ $\Gamma$ may be a relatively open subset
 of the boundary~$\partial M$. If a trace-free
  conformal Killing field~$u$
 vanishes on~$\Gamma$
 then $u\equiv 0$.
\end{Theorem}

In the case of $m=1$ and $\Gamma=\partial M$, Theorem~\ref{1.3}
  follows from~\cite[Proposition~3.3]{[Li]}.

The authors are indebted to the anonymous referee
  for verification of the following fact: Theorem~\ref{1.3} is not valid
  if, in the condition $u\bigr\vert_{\Gamma}=0$, the hypersurface~$\Gamma$
  is replaced by a submanifold of dimension less than~$\dim M-1$.
  We quote: ``The dimension of the space of trace-free
  conformal Killing fields of rank 2
  on~${\mathbb R}^3$ which vanish on a straight line, equals~10,
  i.e., in some sense, Theorem~\ref{1.3} is the best possible result.''

\begin{Corollary}\label{1.4}
 Let $(M,g)$ and $\Gamma$ satisfy the hypotheses of Theorem~$\ref{1.3}.$
 If tensor fields $u\in C^\infty(S^m\tau'_M)$ and
 $v\in C^\infty(S^{m-1}\tau'_M)$ satisfy the equation $du=iv$ and the condition $u\bigr\vert_\Gamma=0$ then there exists a field
 $w\in C^\infty(S^{m-2}\tau'_M)$ such that $u=iw$, $v=dw$, and
 $w\bigr\vert_\Gamma=0$.
\end{Corollary}

 The following definition was introduced in~\cite[\S\,3]{[LPU]}.
 A Riemannian manifold with boundary is called
 {\it conformally rigid}
 if there is no nonzero conformal Killing vector field that
 vanishes on the boundary. Theorem~$\ref{1.3}$ implies
 conformal rigidity of an arbitrary connected Riemannian manifold
 with nonempty boundary. For such a compact manifold~$(M,g)$, Theorem~3.3 of~\cite{[LPU]} can be formulated as follows:
 {\it Every rank~$2$ symmetric tensor field~$f$  on~$M$
 can be uniquely represented in the form}
$$
f_{ij}=\frac {1} {2}
({\nabla}_{\!i}v_j+{\nabla}_{\!j}v_i)+
\lambda g_{ij}+{\tilde f}_{ij},\quad
v|_{\partial M}=0,\quad \mbox{tr}\,\tilde f=0,\quad
\delta \tilde f=0.
$$
We generalize this result to higher rank tensor fields.
  Given a compact~$M$, let $H^k(S^m\tau'_M)$ denote the Hilbert space of
  symmetric tensor fields of rank~$m$ whose components are
  locally square integrable
  together with their partial derivatives up to order~$k$ in an arbitrary coordinate system.

\begin{Theorem}\label{1.5}
 Let $(M,g)$~be a compact connected Riemannian manifold with nonempty boundary.
 Every symmetric tensor field $f\in H^k(S^m\tau'_M)$
 $(m\ge 0$, $k\ge 1)$ can be uniquely represented in the form
\begin{equation}\label{(1.5)}
f=dv+i\lambda+\tilde f,
\end{equation}
where $ v\in H^{k+1}(S^{m-1}\tau'_M) $ satisfies the conditions
\begin{equation}
jv=0,\quad v|_{\partial M}=0,
            \label{(1.6)}
\end{equation}
$ \lambda\in H^k(S^{m-2}\tau'_M) $, and $ \tilde f\in H^k(S^m\tau'_M) $
satisfies the conditions
\begin{equation}
\delta \tilde f=0,\quad j\tilde f=0.
            \label{(1.7)}
\end{equation}
The summands in \eqref{(1.5)} continuously depend on~$f$, i.e.$,$
 the stability estimates
\begin{equation}\label{(1.8)}
\|v\|_{H^{k+1}}\leq C\|f\|_{H^k},\quad
\|\lambda\|_{H^{k}}\leq C\|f\|_{H^k},\quad
\|\tilde f\|_{H^{k}}\leq C\|f\|_{H^k}
\end{equation}
hold with some constant~$C$ independent of~$f$.
\end{Theorem}

In General Relativity,
  conformal Killing tensor fields appear as
  polynomial first integrals of the equation for null geodesics~\cite{[Ge]}.
  Our interest in the conformal Killing
  tensor fields is motivated by
  the following  question from Integral Geometry.

Given a Riemannian manifold~$(M,g)$, let
$$
\Omega M=\{(x,\xi)\mid x\in M,\ \xi\in T_xM,\
|\xi|^2=g_{ij}(x)\xi^i\xi^j=1\}
$$
 denote the unit sphere bundle, and let
$ H:C^\infty(\Omega M)\rightarrow C^\infty(\Omega M) $ be the differentiation
 along the geodesic flow.
 In local coordinates,
\begin{equation}\label{(1.9)}
H=\xi^i\frac {\partial} {\partial x^i}-
\Gamma^i_{jk}(x)\xi^j\xi^k\frac {\partial} {\partial\xi^i},
\end{equation}
 where $\Gamma^i_{jk}$~are the Christoffel symbols. As is seen from~\eqref{(1.9)},
if the function~$U(x,\xi)$ polynomially depends  on~$\xi$ then~$HU$
 is also a polynomial in~$\xi$. More precisely, for
 $u\in C^\infty(S^{m-1}\tau'_M)$,
\begin{equation}\label{(1.10)}
H(u_{i_1\dots
i_{m-1}}(x)\xi^{i_1}\dots\xi^{i_{m-1}})= (du)_{i_1\dots
i_{m}}(x)\xi^{i_1}\dots\xi^{i_{m}}.
\end{equation}

 The question on validity of the
   converse statement is very important:
  {\it Is it true that every solution to the
   boundary value problem}
\begin{equation}\label{(1.11)}
 HU=v_{i_1\dots
i_{m}}(x)\xi^{i_1}\dots\xi^{i_{m}} \quad \mbox{on}\quad \Omega M,
\end{equation}
\begin{equation}\label{(1.12)}
U|_{\partial(\Omega M)}=0
\end{equation}
{\it is a homogeneous polynomial of degree~$m-1$ in~$\xi$}?
 The question is equivalent to the problem of inversion of the ray transform
 (see~\cite[Ch.\,1]{[mb]} for a detailed discussion).
 The question is open in the general case and the positive answer is obtained
 only under certain curvature conditions.

Consider the following weaker version of the latter question.
   {\it Assume a solution $U$ to the boundary value
   problem \eqref{(1.11)}--\eqref{(1.12)} to depend polynomially
   on~$\xi$. Is $U$ a restriction to~$\Omega M$ of a homogeneous polynomial
   of degree~$m-1$}?
   The question is not trivial since  polynomials of different degrees
   can have the same restriction to~$\Omega M$ in view of the
   identity $g_{ij}\xi^i\xi^j\bigr\vert_{\Omega M}=1$. The positive answer
   to this question can be easily obtained from Theorem~\ref{1.3}
   even if~\eqref{(1.12)} is replaced by the weaker condition
\begin{equation}\label{(1.13)}
U(x,\xi)|_{x\in\Gamma}=0,
\end{equation}
where $\Gamma$~ is a relatively open subset of~$\partial M$. Indeed, assume a
 solution~$U(x,\xi)$ to the problem~\eqref{(1.11)} and \eqref{(1.13)}  to be
 a homogeneous polynomial in~$\xi$ of degree $m+2k-1$
 (the case of a nonhomogeneous polynomial can be easily reduced to the considered one).
 This means the existence of
 $u\in C^\infty(S^{m+2k-1}\tau'_M)$ such that
$$
U(x,\xi)=u_{i_1\dots
i_{m+2k-1}}(x)\xi^{i_1}\dots\xi^{i_{m+2k-1}} \quad \mbox{on}\quad
\Omega M,\quad u|_\Gamma=0.
$$
By \eqref{(1.10)}, equation~\eqref{(1.11)} takes the form
\begin{equation}\label{(1.14)}
du=i^kv.
\end{equation}
 Applying Corollary~\ref{1.4}, we find
 $w\in C^\infty(S^{m+2k-3}\tau'_M)$ such that $u=iw$ and $w\bigr\vert_\Gamma=0$.
 Equation~\eqref{(1.14)} can be written in terms of~$w$ as
 $i(dw)=i(i^{k-1}v)$. Since $i$~is a monomorphism, this implies
$$
dw=i^{k-1}v.
$$
Repeating this argument by induction in~$k$, we find the
 field $\widetilde w\in C^\infty(S^{m-1}\tau'_M)$ such that
 $u=i^k\widetilde w$. This means that $U(x,\xi)\bigr\vert_{\Omega M}$
 coincides with the homogeneous polynomial
 ${\widetilde w}_{i_1\dots i_{m-1}}(x)\xi^{i_1}\cdots\xi^{i_{m-1}}$ of
 degree~$m-1$.

\bigskip

 The questions under consideration are also important in the case of closed manifolds
  (i.e., compact manifolds with no boundary).

\begin{Theorem}\label{negative}
If $(M,g)$~is a closed Riemannian manifold of dimension $n\ge 2$
 of nonpo\-si\-tive sectional curvature then every trace-free
 conformal Killing symmetric
 tensor field $u$ on~$M$ is absolutely parallel, i.e.,
  $\nabla u=0$,
 and every symmetric Killing tensor field is
 absolutely parallel.

 In addition, if~$M$ is connected and there is a point $x_0\in M$
 such that all sectional curvatures at $x_0$ are negative then there is no
 nonzero trace-free conformal Killing
 symmetric tensor field of any rank and every symmetric
 Killing tensor field is of the form~$cg^k$ for some constant~$c$.
\end{Theorem}

  The classical theorem by Bochner-Yano states:
 there is no nontrivial conformal Killing vector field
  on a closed Riemannian manifold of negative Ricci curvature
  \cite[Theorem 2.14]{[JB]}.
  Theorem 1.6 generalizes the last statement to arbitrary rank
  tensor fields,
  however, under the stronger hypothesis:
  The requirement of negative Ricci curvature is replaced by
  the requirement of negative sectional curvature.

\begin{Theorem}\label{focal}
Let $(M,g)$~be a closed Riemannian manifold of dimension $n\ge 2$
 without conjugate points. A vector field on~$M$ is
 conformal Killing if and only if it is a Killing vector field.
 A trace-free tensor field of
 rank~$2$ is
 conformal Killing if and only if
  it is the trace-free part of some Killing field.
  In addition, if the geodesic flow has a dense orbit in~$\Omega M$
 then there are neither nontrivial
 conformal Killing vector fields nor nontrivial trace-free
 conformal Killing fields of rank~$2$.
\end{Theorem}

 The natural assumption is that, under hypotheses of Theorem~\ref{focal},
 similar statements are valid for higher rank tensor fields. In the case of $\dim M=2$, this easily follows from the
 uniformization theorem,
 invariance of the definition of conformal Killing tensor fields
 with respect to a conformal change of the metric,
  and from Theorem~\ref{negative}.

The following fact is well known: If the geodesic flow has
  a dense orbit in~$\Omega M$ then
 (regardless of the dimension of $M$)
 every symmetric Killing tensor field is of
  the form~$cg^k$ with some constant~$c$ (see \cite{[Cr-Sh]}).

\bigskip

The rest of the paper is organized as follows. Section 2 contains preliminaries
   from  algebra of symmetric tensor fields. In particular,
  after deriving a commutation formula for the operators~$i$ and~$j$,
  we show that the singular decomposition of the operator~$ji$
  corresponds to the decomposition of polynomials
  in spherical harmonics.

 In~Section\,3, we introduce the differential operators~$d$ and~$\delta$
 on symmetric tensor fields, prove some commutation
 formulas for these operators, and obtain some useful
 propositions. In this section, we derive Corollary~\ref{1.2} of Theorem~\ref{1.1}.

   Sections 4 and 5 contain the proofs of Theorems~\ref{1.3} and~\ref{1.5},
  respectively.  These proofs are essentially based on Theorem~\ref{1.1}.

  In Section 6, we give the proofs of Theorems~\ref{negative} and \ref{focal}.
  It should be mentioned that this section differs from the others by the nature
  of the methods. Namely, here we use semi-basic tensor fields and the estimates
  for a solution to the kinetic equation which are based on Pestov's identity.
  We cannot present all necessary definitions here
  because of the volume limitation, so we refer the reader
  to \cite[Chapter 3]{[mb]} for details.
  This does not concern other sections
  since they are independent of Section 6.

 In Section 7, we discuss higher order
 differential operators on tensor fields paying a particular attention
 to the principle parts of operators.

 In Section 8, we introduce the Laplace operator on symmetric tensor fields and
 prove some commutation formulas for powers of the Laplacian which are
 needed to prove Theorem \ref{1.1}.

 The proof of Theorem~\ref{1.1} in the case of $\dim M\geq 3$ is
 presented in Section\,9.

 In Section 10, we derive the equations that relate the Fourier coefficients of a solution
 to the kinetic equation, to the Fourier coefficients of the right-hand side. We need these equations
 to prove Theorem \ref{1.1} in the two-dimensional case.
 These equations are also of some independent interest since they
 constitute the basis of the so-called method of spherical harmonics
 for the numerical solution of the kinetic equation and the related linear
 transport equation.
 In the literature on the method of spherical harmonics, several versions of
 the equations are presented for different particular geometries (see \cite{[CZ]}).
 However, the invariant form of the equations, as
 presented in Theorem~10.2, was probably unknown before.

 The final Section 11 contains the proof of Theorem~\ref{1.1} in
 the two-dimensional case.

\section{Algebra of symmetric tensors}

We use the standard terminology of vector bundle theory.
   For a smooth manifold~$N$, we denote the algebra of smooth real
   functions on~$N$ by $C^\infty(N)$. If $\xi=(E,\pi,N)$~ is a smooth vector bundle and
   $U\subset N$ is an open set then $C^\infty(\xi;U)$ denotes the $C^\infty(U)$-module
   of smooth sections of~$\xi$ over~$U$, and
   $C^\infty_0(\xi;U)$ denotes the submodule of compactly supported sections.
   We often reduce the notation $C^\infty(\xi;N)$ and $C^\infty_0(\xi;N)$
   to $C^\infty(\xi)$ and $C^\infty_0(\xi)$, respectively.
   We  deal here only with finite dimensional bundles with just one
   exception: Sometimes, we consider a graded vector bundle
     $\xi^*=\oplus_{m=0}^\infty\xi^m$,
   where each summand $\xi^m$ has a finite dimension.
   Such an object can be thought as a sequence of finite-dimensional
   bundles. If $\eta=\oplus_{m=0}^\infty\eta^m$ is another graded bundle
   and $A\in\Hom(\xi,\eta)$ then $A_m$ denotes the restriction of~$A$
   to~$\xi^m$.  We say $A$ has a degree~$k$ if
   $A(\xi^m)\subset\eta^{m+k}$.

Let $(M,g)$~be a smooth Riemannian manifold of dimension $n\ge 2$.
   By $\tau_M=(TM,\pi,M)$ and $\tau'_M=(T'M,\pi,M)$, we denote the tangent bundle
   and the cotangent bundle, respectively.
   We often shorten these notation to $\tau=(T,\pi,M)$ and $\tau'=(T',\pi,M)$.
   Let $\otimes^m\tau'=(\otimes^m T',\pi,M)$~be the bundle of real covariant
   tensors of rank~$m$ and let $S^m\tau'=(S^mT',\pi,M)$~be its subbundle
   consisting of the~symmetric tensors.
   There is the natural projection $\sigma\in\Hom(\otimes^m\tau',S^m\tau')$ acting
   as follows:
\begin{equation}\label{(2.1)}
\sigma(v_1\otimes\dots\otimes v_m)=
\frac {1} {m!}\sum\limits_{\pi\in\Pi_m}
v_{\pi(1)}\otimes\dots\otimes v_{\pi(m)},
\end{equation}
where $ \Pi_m $ is the group of permutations of the set $\{1,\dots,m\}$.
 Note that $S^1\tau'{=}\nobreak\tau'$ and $S^0\tau'=M\times\mathbb R$.
 It is convenient to assume
 $\otimes^m\tau'=S^m\tau'=0$ for $m<0$.
 For a point $x\in M$, let $T_x$, $T'_x$, $\otimes^mT'_x$, and $S^mT'_x$
 be the fibers over~$x$ of the corresponding bundles.

For $ u\in S^mT'_x $ and $ v\in S^lT'_x $, the {\it symmetric product} is defined by
$ uv=\sigma(u\otimes v) $. Thus,
$ S^*\tau'=\oplus_{m=0}^\infty S^m\tau' $ becomes a bundle of commutative graded algebras.

We will extensively use the coordinate representation of tensors.
  If $(x^1,\dots,x^n)$~is a local coordinate system in the domain
  $U\subset M$ then every tensor field $u\in C^\infty(\otimes^m\tau';U)$
  is uniquely represented in the form
\begin{equation}\label{(2.2)}
u=u_{i_1\dots i_{m}}dx^{i_1}\otimes\dots\otimes dx^{i_m},
\end{equation}
where the functions $u_{i_1\dots i_m}\in C^\infty(U)$ are the
  {\it covariant components}
  of the field~$u$ in this coordinate system.
  In~\eqref{(2.2)} and below, we use the Einstein rule:
  The summation from 1 to~$n$ is assumed over an index repeated in the multivariate subscript and superscript of a monomial.
  Assuming the choice of a coordinate system to be clear from the context, we reduce formula
  \eqref{(2.2)} to
\begin{equation}\label{(2.3)}
u=(u_{i_1\dots i_{m}}).
\end{equation}
For $x\in U$ and $u\in\otimes^mT'_x$, formulas~\eqref{(2.2)}
  and~\eqref{(2.3)} also make sense but the components
   are real numbers in this case.
  {\it Contravariant components}
  are defined by
$$
u^{i_1\dots i_{m}}=
g^{i_1j_1}\dots g^{i_mj_m}u_{j_1\dots j_{m}},
$$
where $ (g^{ij}) $ is the inverse matrix to $ (g_{ij}) $.

A tensor $u=(u_{i_1\dots i_{m}})\in\otimes^mT'_x$ belongs
 to~$S^mT'_x$ if and only if its covariant and (or)
 contravariant components are symmetric with respect to all indices.
 We will also consider partially symmetric tensors.
 The partial symmetry of a tensor is denoted by
\begin{equation}
\mbox{sym}\ u_{i_1\dots i_{k}j_1\dots j_l}:
(i_1\dots i_{k-1})i_k(j_1\dots j_{l-1})j_l.
            \label{(2.4)}
\end{equation}
This means the tensor $(u_{i_1\dots i_kj_1\dots j_l})$ is symmetric
 in each group of indices in parentheses on the right-hand side of~\eqref{(2.4)}.
Along with the full symmetrization $\sigma$, we will use partial symmetrization operators that are defined in coordinates by
$$
\sigma(i_1\dots i_m)
u_{i_1\dots i_mj_1\dots j_l}=
\frac {1} {m!}\sum\limits_{\pi\in\Pi_m}
u_{i_{\pi(1)}\dots i_{\pi(m)}j_1\dots j_l}.
$$

\begin{Lemma}\label{2.1}
Let $m\ge 1$$,$ $p\ge 1$$,$ and $x\in M$. For every tensor
 $f\in\otimes^{2m+p}T'_x$ possessing the symmetry
\begin{equation}
\mbox{\rm sym}\ f_{i_1\dots i_mj_1\dots j_pk_1\dots k_m}:
(i_1\dots i_mj_1\dots j_p)(k_1\dots k_m),
            \label{(2.5)}
\end{equation}
 there exists a unique solution to the equation
\begin{equation}
\sigma(i_1\dots i_mj_1\dots j_p)
u_{i_1\dots i_mj_1\dots j_pk_1\dots k_m}=
f_{i_1\dots i_mj_1\dots j_pk_1\dots k_m},
            \label{(2.6)}
\end{equation}
possessing the symmetry
\begin{equation}
\mbox{\rm sym}\ u_{i_1\dots i_mj_1\dots j_pk_1\dots k_m}:
(i_1\dots i_m)(j_1\dots j_pk_1\dots k_m).
            \label{(2.7)}
\end{equation}
The solution is expressed by the formula
\begin{multline}
u_{i_1\dots i_mj_1\dots j_pk_1\dots k_m}=
\sigma(i_1\dots i_m)\sigma(j_1\dots j_pk_1\dots k_m)
\\
\times\sum\limits_{l=0}^{m}(-1)^l
{p+l-1\choose l}{m+p\choose m-l}
f_{i_1\dots i_{m-l}j_1\dots j_pk_1\dots k_mi_{m-l+1}\dots i_m},
            \label{(2.8)}
\end{multline}
where $ {k\choose l}=\frac {k!} {l!(k-l)!} $ are the binomial
 coefficients.
\end{Lemma}

Since this is a purely algebraic statement, it suffices to prove it
 in the case of $M={\mathbb R}^n $. For $p=1$, the statement is proved
 in~\cite[\S\,2.4]{[mb]}. For an arbitrary~$p$, the proof is quite similar.
 The idea of the proof is as follows. Since the dimension of the space
 of tensors possessing symmetry~\eqref{(2.5)} is equal to the dimension
 of the space of tensors possessing symmetry~\eqref{(2.7)},
 it suffices to verify the equation obtained by substituting
\eqref{(2.8)} into~\eqref{(2.6)}.
 This verification can be done by straightforward calculations that are omitted.

There is a natural inner product on~$S^mT'_x$ defined in coordinates by
\begin{equation}
\langle
u,v\rangle= u_{i_m\dots i_m}v^{i_m\dots i_m}.
                                    \label{(2.9)}
\end{equation}
 We extend the product to $S^*T'_x=\oplus_{m=0}^\infty S^mT'_x$
assuming $S^mT'_x$ and $S^lT'_x$ to be orthogonal to each other for $m\neq l$.
  The product smoothly depends on~$x$. Therefore, $S^*\tau'$ obtains the
  structure of a Riemannian vector bundle. So we can introduce
 the $L^2$-product on $C^\infty_0(S^*\tau')$
 as follows:
\begin{equation}
(u,v)_{L^2}=\int\limits_{M}\langle u(x),v(x)\rangle dV(x),
            \label{(2.10)}
\end{equation}
where $ dV $ is the Riemannian volume form.

For $u\in S^*T'_x$, let $i_u:S^*T'_x\rightarrow S^*T'_x$ be
  the operator of symmetric multiplication by~$u$, i.e.,
 $i_uv=uv$, and let~$j_u$ be the adjoint operator of $i_u$.
 These operators are expressed by the formulas
$$
(i_uv)_{i_1\dots i_{m+l}}=
\sigma(i_1\dots i_{m+l})(u_{i_1\dots i_m}v_{i_{m+1}\dots i_{m+l}}),
$$
$$
(j_uv)_{i_1\dots i_{m-l}}=
v_{i_1\dots i_l}u^{i_{l-m+1}\dots i_{l}}
$$
for $ u\in S^mT'_x $ and $ v\in S^lT'_x $. The second formula makes sense only for $m\le l$.
If $ m>l $ then $ j_uv=0 $.
For $u\in C^\infty(S^*\tau')$,
 the operators $i_u,j_u\in\Hom(S^*\tau',S^*\tau')$
 are defined by
$ i_u(x)=i_{u(x)} $ and $ j_u(x)=j_{u(x)} $. The operators~$i_g$ and~$j_g $ are of a
particular importance in the present article, so we distinguish them by introducing the notation $i=i_g$
 and $j=j_g$.
 These operators were already used in Introduction.

\begin{Lemma}\label{2.2}
For $m\ge 0$ and $k\ge 1$$,$
 the following commutation formula holds
  on~$S^m\tau'$\dv
$$
ji^k=
\frac {2k(n+2m+2k-2)} {(m+2k-1)(m+2k)}i^{k-1}+
\frac {m(m-1)} {(m+2k-1)(m+2k)}i^kj.
$$
\end{Lemma}

 In the case of $k=1$, the formula has the form
\begin{equation}
ji=
\frac {2(n+2m)} {(m+1)(m+2)}E+
\frac {m(m-1)} {(m+1)(m+2)}ij,
            \label{(2.11)}
\end{equation}
where $ E $ the identity operator. The latter formula is proved by a straightforward calculation in coordinates which is omitted.
The general
 case easily follows from~\eqref{(2.11)} with the help of induction in $k$.

\begin{Lemma}\label{2.3}
For an arbitrary integer $m\ge 0$$,$ the following decomposition formula holds\dv
\begin{equation}
S^m\tau'=\bigoplus\limits_{k=0}^{[m/2]}
i^k(\mbox{\rm Ker}\,j_{m-2k}),
            \label{(2.12)}
\end{equation}
where $[m/2]$~is the integer part of~$m/2$$,$ and $\Ker j_{m-2k}$~is
 the kernel of the restriction $j_{m-2k}$ of the operator~$j$
 to $S^{m-2k}\tau'$.
  Each summand of the decomposition is a subbundle in~$S^m\tau'$
  and the summands are orthogonal to each other.
  The operator~$i$ is a monomorphism and its range is related to
  decomposition~\eqref{(2.12)} by the formula
\begin{equation}
\mbox{\rm Ran}\,i_{m-2}=\bigoplus\limits_{k=1}^{[m/2]}
i^k(\mbox{\rm Ker}\,j_{m-2k}).
            \label{(2.13)}
\end{equation}
The product  $ji$ is a self-adjoint and positive definite operator.
 Each summand of~\eqref{(2.12)} is a proper subspace
 of the operator~$ji$ associated with the eigenvalue
$$
\lambda_k=\frac {2(k+1)(n+2m-2k)} {(m+1)(m+2)}.
$$
The dimension of the summand equals $\alpha(m-2k)-\alpha(m-2k-2)$,
 where $\alpha(m)=\alpha(m,n)={n+m-1\choose m}$ is the
 dimension of~$S^m\tau'$.
\end{Lemma}

\begin{proof}
 The operator $ij$ is self-adjoint and nonnegative since it is the product
 of two operators that are dual to each other.
 Therefore,~\eqref{(2.11)} implies that $ji$~is a positive self-adjoint operator.
 Hence, $i$~is a monomorphism and the orthogonal decomposition
\begin{equation}
S^*\tau'=\mbox{\rm Ker}\,j\oplus\mbox{Ran}\,i
            \label{(2.14)}
\end{equation}
 holds with the summands on the right-hand side being sub-bundles of the left-hand side.

Let $u\in\Ker j_{m-2k}$. By Lemma~\ref{2.2}, we get
   $(ji)(i^ku)=\lambda_ki^ku$. Therefore, each summand in~\eqref{(2.12)}
   is the
   eigen-subspace of the operator~$ji$ associated
  with the eigenvalue~$\lambda_k$. Since all~$\lambda_k$ are different,
  all summands are orthogonal to each other.
 The injectivity of  $i$~ and equation~\eqref{(2.14)} imply
$$
\mbox{dim}\,[i^k(\mbox{\rm Ker}\,j_{m-2k}]=
\mbox{dim}\,(\mbox{\rm Ker}\,j_{m-2k})=
\alpha(m-2k)-\alpha(m-2k-2).
$$
Equality~\eqref{(2.12)} is now proved by comparing
 dimensions of spaces on both sides of the equality.
\end{proof}

\bigskip

 Let $p\in\Hom(S^*\tau',\Ker j)$ and
  $q\in\Hom(S^*\tau',\Ran i)$ be the orthogonal projections onto the
  summands of \eqref{(2.14)}. One easily checks the equality
\begin{equation}
q=i(ji)^{-1}j.
            \label{(2.15)}
\end{equation}

Decomposition~\eqref{(2.12)} is closely related to the expansion of
  functions on the sphere
  in Fourier series
  in spherical harmonics.
  In order to explain the relationship, we introduce some notation.

If  $(x^1,\dots,x^n)$ is a local coordinate system with the domain~$U$ then, for $x\in U$, a vector  $\xi\in T_x$ is uniquely represented as
   $\xi=\xi^i\frac{\partial}{\partial x^i}(x)$.
   The functions $x^i,\xi^i$ $(i=1,\dots,n)$ form a local
  coordinate system on the manifold~$T$ with the
  domain~$\pi^{-1}(U)$,
  where $\pi$~is the projection of the tangent bundle. Strictly speaking,
  we should write $x^i\circ\pi$ instead of~$x^i$.  We use the shorter
  notation~$x^i$ and hope it will not cause any ambiguity.
  Thus, every function $\varphi\in C^\infty\big(\pi^{-1}(U)\big)$ can be
  written in coordinates as follows:
$$
\varphi=\varphi(x^1,\dots,x^n,\xi^1,\dots,\xi^n)\quad
(x\in U,\xi\in T_x).
$$
 Since $T_x$ has the structure of an Euclidean space, the
  Laplace operator
 $ {\Deltav}_{x}:C^\infty(T_x)\rightarrow C^\infty(T_x) $ is
 well defined. It smoothly depends on~$x$, and therefore,
 defines an operator~$\Deltav:C^\infty(T)\rightarrow C^\infty(T)$.
 \big(Warning: Do not mix up $C^\infty(T)$ and $C^\infty(\tau)$!\big)
 The operator is written in coordinates as
$$
\Deltav\varphi(x,\xi)=g^{ij}(x)
\frac {\partial^2\varphi(x,\xi)} {\partial\xi^i\partial\xi^j}
$$
 and is called the
 {\it vertical}
 (or {\it fiberwise})
 Laplacian.

The embedding $\varkappa_x:S^*T'_x\rightarrow C^\infty(T_x)$
 is defined by
$$
(\varkappa_xu)(\xi)=u_{i_1\dots i_m}\xi^{i_1}\dots\xi^{i_m}
$$
for $u\in S^mT'_x$. It smoothly depends on $x$, and hence,
 defines an embedding  $\varkappa:C^\infty(S^*\tau')\rightarrow C^\infty(T)$ by the formula:
$$
(\varkappa u)(x,\xi)=u_{i_1\dots i_m}(x)\xi^{i_1}\dots\xi^{i_m}
$$
for $ u\in C^\infty(S^m\tau') $. Thus,~$\varkappa$ identifies  rank~$m$ symmetric tensor fields
with homogeneous polynomials
 (with respect to~$\xi$) of degree~$m$ on~$T$.

Let $\Omega=\Omega M$~be the submanifold of $T$ which consists of the unit vectors,
 and let $\Omega_x=\Omega\cap T_x$ be the unit sphere in~$T_x$.
 Given $u\in S^*T'_x$, denote the
 restriction of~$\varkappa_xu$ to~$\Omega_x$ by~$\lambda_xu$ . The operator
 $\lambda_x:S^*T'_x\rightarrow C^\infty(\Omega_x)$
 smoothly depends on~$x$ and defines an operator
$ \lambda:C^\infty(S^*\tau')\rightarrow C^\infty(\Omega) $.

We introduce an inner product
 $\langle\boldsymbol{\cdot}\,,\boldsymbol{\cdot}\rangle_\omega$
 on the space $C^\infty(\Omega_x)$ by the formula
$$
\langle\varphi,\psi\rangle_\omega =
\int\limits_{\Omega_x}\varphi(\xi)\psi(\xi)d\omega(\xi),
$$
where $d\omega$~is the volume form on~$\Omega_x$ induced
 by the metric~$g$. The index~$\omega$ is used in the notation
 in order to distinguish this product from the product defined by \eqref{(2.9)}.

\begin{Lemma}\label{2.4}
The following equalities hold on~$S^mT'_x$\dv
\begin{equation}
\lambda_xi=\lambda_x,
            \label{(2.16)}
\end{equation}
\begin{equation}
m(m-1)\varkappa_xj={\Deltav}_{x}\varkappa_x.
            \label{(2.17)}
\end{equation}
For $u,v\in S^mT'_x$ such that $ju=jv=0$$,$
 the following equality holds$:$
\begin{equation}
\langle\lambda_x u,\lambda_x v\rangle_\omega =
\frac {m!\pi^{n/2}} {2^{m-1}\Gamma(n/2+m)}
\langle u,v\rangle.
            \label{(2.18)}
\end{equation}
\end{Lemma}

\begin{proof}
 For $u\in S^mT'_x$ and $\xi\in\Omega_x$, we have
$$
(\lambda_xi u)(\xi)=
\Big(\sigma(i_1\dots i_{m+2})(u_{i_1\dots i_m}
g_{i_{m+1}i_{m+2}})\Big)\xi^{i_1}\dots\xi^{i_{m+2}}.
$$
 Since the product
  $\xi^{i_1}\cdots\xi^{i_{m+2}}$
 is symmetric with respect to
 $(i_1,\dots,i_{m+2})$,
 we can omit the symmetrization
 $\sigma(i_1\dots i_{m+2})$ here.
 So we get
$$
(\lambda_xiu)(\xi)=
(u_{i_1\dots i_m}\xi^{i_1}\dots\xi^{i_m})
(g_{i_{m+1}i_{m+2}}\xi^{i_{m+1}}\xi^{i_{m+2}})=
u_{i_1\dots i_m}\xi^{i_1}\dots\xi^{i_m}
$$
because $g_{i_{m+1}i_{m+2}}\xi^{i_{m+1}}\xi^{i_{m+2}}=1$
 on~$\Omega_x$.
 By definition, the right-hand side of the last formula equals
  $(\lambda_xu)(\xi)$. This proves~\eqref{(2.16)}.
Formula~\eqref{(2.17)}
 is also proved by direct calculations:
$$
\begin{aligned}
({\Deltav}_{x}\varkappa_xu)(\xi)&=g^{ij}
\frac {\partial^2} {\partial\xi^i\partial\xi^j}
\left(u_{i_1\dots i_m}\xi^{i_1}\dots\xi^{i_m}\right)
\\
&=
m(m-1)g^{ij}
u_{i_1\dots i_{m-2}ij}\xi^{i_1}\dots\xi^{i_{m-2}}
=m(m-1)(\varkappa_xju)(\xi).
\end{aligned}
$$

It suffices to prove \eqref{(2.18)} for $u=v$.
 The polynomial $\varkappa_xu$ can be written in the two ways:
\begin{equation}
(\varkappa_xu)(\xi)=
u_{i_1\dots i_m}\xi^{i_1}\dots\xi^{i_m}=
\sum\limits_{|\alpha|=m}u_\alpha\xi^\alpha,
            \label{(2.19)}
\end{equation}
where $\alpha=(\alpha_1,\dots,\alpha_n)$~is a multiindex.
 The coefficients of~\eqref{(2.19)} are related by the equality
  $u_\alpha=\frac{m!}{\alpha!}u_{i(\alpha)}$, with
 $i(\alpha)=(1\dots 1\,2\dots 2\dots n\dots n)$ where $1$
 is repeated~$\alpha_1$ times,
 $2$~is repeated $\alpha_2$ times, etc.

We choose the coordinates in a neighborhood of~$x$
 so that
  $g_{ij}(x)=\delta_{ij}$, where $(\delta_{ij})$~is
 the Kronecker symbol. Then
 \begin{equation}
\langle u,u\rangle =\frac {1} {m!}
\sum\limits_{|\alpha|=m}\alpha!|u_\alpha|^2.
            \label{(2.20)}
\end{equation}
Indeed,
$$
\langle u,u\rangle =
\sum\limits_{i_1,\dots i_m=1}^{n}
|u_{i_1\dots i_m}|^2=
\sum\limits_{|\alpha|=m}\frac {m!} {\alpha!}|u_{i(\alpha)}|^2.
$$
Taking the relations $u_{i(\alpha)}=\alpha!u_\alpha/m!$ into account,
 we obtain~\eqref{(2.20)}.

For $|\alpha|=m$, formula \eqref{(2.19)} implies
\begin{equation}
\partial^\alpha_\xi(\varkappa_xu)=\alpha!u_\alpha.
            \label{(2.21)}
\end{equation}
Therefore equality~\eqref{(2.20)} can be written as
\begin{equation}
\langle u,u\rangle =\frac {1} {m!}
\sum\limits_{|\alpha|=m}
\frac {1} {\alpha!}|\partial^\alpha_\xi(\varkappa_xu)|^2.
            \label{(2.22)}
\end{equation}

In virtue of \eqref{(2.17)}, the condition $ju=0$ means $\lambda_xu$
 is a spherical harmonics of degree $m$. Ass well known,
 the spherical harmonics satisfy the equality
\begin{equation}
A_{m,m}=
\frac {2^m\Gamma(n/2+m)\Gamma(m+1)} {\Gamma(n/2)}A_{0,m},
            \label{(2.23)}
\end{equation}
(for example, see \cite[Lemma~XI.1]{[Sob]}),
    where
\begin{equation}
A_{0,m}=
\int\limits_{\Omega_x}|\lambda_xu(\xi)|^2d\omega(\xi)
=\langle\lambda_xu,\lambda_xu\rangle_\omega,
            \label{(2.24)}
\end{equation}
$$
A_{m,m}=\sum\limits_{|\alpha|=m}
\frac {m!} {\alpha!}
\int\limits_{\Omega_x}|\partial^\alpha_\xi(\varkappa_xu)|^2
d\omega(\xi).
$$
By \eqref{(2.21)}, the last integrand is constant and the last formula gives
$$
A_{m,m}=m!\omega_n\sum\limits_{|\alpha|=m}
\alpha!|u_\alpha|^2,
$$
where $\omega_n$~is the volume of the unit sphere
 in~${\mathbb R}^n$.
  Together with~\eqref{(2.20)}, the last formula implies
\begin{equation}
A_{m,m}=(m!)^2\omega_n\langle u,u\rangle.
            \label{(2.25)}
\end{equation}

Finally, substituting~\eqref{(2.24)}, \eqref{(2.25)}, and
 $\omega_n\,{=}\,2\pi^{n/2}/\Gamma(n/2)$ into~\eqref{(2.23)},
 we obtain~\eqref{(2.18)}.
 \end{proof}


According to Lemma~\ref{2.4}, the operator~$\lambda_x$
 isomorphically maps the subspace $\Ker j_m\subset S^mT'_x$ onto the space
 of spherical harmonics of degree~$m$ on~$\Omega_x$. Moreover,
 this isomorphism is an isometry up to a constant factor
 if~$\Ker j_m$ is equipped with the inner product
   $\langle\boldsymbol{\cdot}\,,\boldsymbol{\cdot}\rangle$
     and the space of harmonics is equipped with the product
 $\langle\boldsymbol{\cdot}\,,\boldsymbol{\cdot}\rangle_\omega$.

As well known (for example, see~\cite{[Sob]}), spherical harmonics
 of different degrees are orthogonal to each other and every
 function $\varphi\in C^\infty(\Omega_x)$ can be expanded
 in the Fourier series in spherical harmonics of different degrees.
 The series converges absolutely and uniformly.
 The expansion smoothly depends on the point~$x$, and
  we arrive to the following statement.

\begin{Lemma}\label{2.5}
Every function
 $\varphi\in C^\infty(\Omega)$
 can be uniquely represented by the series
\begin{equation}
\varphi=\sum\limits_{m=0}^{\infty}\lambda u_m,
            \label{(2.26)}
\end{equation}
where $u_m\in C^\infty(S^m\tau')$ satisfy the condition $ju_m=0$.
 The series converges absolutely and uniformly on each compact
 subset of~$\Omega$.
\end{Lemma}

By Lemma~2.3, a tensor field
  $u\in C^\infty(S^m\tau')$ is
  uniquely represented in the form
\begin{equation}
u=\sum\limits_{k=0}^{[m/2]}i^ku_{m-2k},
            \label{(2.27)}
\end{equation}
where every
 $u_{m-2k}\in C^\infty(S^{m-2k}\tau')$
 satisfies
 $ju_{m-2k}=\nobreak 0$. Expansion~\eqref{(2.27)}
 coincides with the Fourier series~\eqref{(2.26)} of the function
 $\lambda u\in C^\infty(\Omega)$.
 More precisely, the Fourier series of the function~$\lambda u$
 has a finite number of the summands, namely,
$$
\lambda u=\sum\limits_{k=0}^{[m/2]}\lambda u_{m-2k}
$$
where the tensor fields $u_{m-2k}$ are the same as in~\eqref{(2.27)}.

\section{The operators $ d $ and $ \delta $}

Given a Riemannian manifold~$(M,g)$, let
$ {\nabla}:C^\infty(\otimes^m\tau')\rightarrow
C^\infty(\otimes^{m+1}\tau') $ be the covariant
 differentiation with respect to the Levi-Civita connection.
For a tensor field $ u=(u_{i_1\dots i_m}) $, higher order covariant derivatives are denoted by
$ {\nabla}^ku=({\nabla}_{\!j_1\dots j_k}u_{i_1\dots i_m}) $.
The
 {\it inner derivative}
 $d:C^\infty(S^*\tau')\rightarrow C^\infty(S^*\tau')$
 is defined by
 $d=\sigma{\nabla}$. The
 {\it divergence\/}
 $\delta:C^\infty(S^*\tau')\rightarrow C^\infty(S^*\tau')$
 is defined in local coordinates as
 $(\delta u)_{i_1\dots i_{m-1}}=g^{jk}\nabla_j u_{ki_1\dots i_{m-1}}$.
 These~$d$ and $\delta$ are the first order differential operators of
 degree~$1$ and $-1$, respectively.
 They were already mentioned in Introduction.

\begin{Theorem} \label{3.1}
The operators $d$ and $-\delta$ are dual
  to each other with respect to the $L^2$-product of symmetric tensor fields.
 Moreover$,$ for a compact manifold~$M$ with boundary$,$  Green's formula
$$
\int\limits_{M}(
\langle du,v\rangle +\langle u,\delta v\rangle)dV=
\int\limits_{\partial M}
\langle i_\nu u,v\rangle dV'
$$
holds for $u,v\in C^\infty(S^*\tau')$$,$ where $\nu$~is the unit
 outward normal to the boundary$,$ and~$dV$ and $dV'$ are
 the Riemannian volume forms of~$M$ and of $\partial M$$,$ respectively.
\end{Theorem}

The proof is given in~\cite[\S\,3.3]{[mb]}.

\begin{Lemma} \label{3.2}
The following equalities hold on $C^\infty(S^*\tau')$\dv
\begin{alignat}2
 &id=di, &&\qquad j\delta=\delta j,
 \label{(3.1)}\\
 &pdp=pd,&&\qquad p\delta p=\delta p,
 \label{(3.2)}\\
 &qdq=dq,&&\qquad q\delta q=q\delta,
 \label{(3.3)}\\
 &pdq=0, &&\qquad q\delta p=0.
 \label{(3.4)}
 \end{alignat}
\end{Lemma}

 \begin{proof}
 These formulas are written in pairs every of which is formed by
 two relatively dual relationships. It suffices to prove one formula
 in each pair. The second formula in~\eqref{(3.1)} can be proved by
 direct calculation in coordinates and we omit it.
 The first formula in~\eqref{(3.3)} is derived from~\eqref{(2.15)}  and~\eqref{(3.1)} as follows:
$$
qdq=i(ji)^{-1}jdi(ji)^{-1}j=i(ji)^{-1}jid(ji)^{-1}j=
id(ji)^{-1}j=di(ji)^{-1}j=dq.
$$
 Formula~\eqref{(3.2)} is obtained from~\eqref{(3.3)} by the
 substitution $q=E-p$. The left multiplication of the first
 formula of~\eqref{(3.3)} by~$p$ implies the first
 formula of~\eqref{(3.4)}.
\end{proof}

\begin{Lemma} \label{3.3}
The following equalities hold on $C^\infty(S^m\tau')$\dv
\begin{equation}
\delta i=\frac {2} {m+2}d+\frac {m} {m+2}i\delta,
            \label{(3.5)}
\end{equation}
\begin{equation}
jd=\frac {2} {m+1}\delta+\frac {m-1} {m+1}dj.
            \label{(3.6)}
\end{equation}
\end{Lemma}

\begin{proof}
 It suffices to prove~\eqref{(3.5)} since~\eqref{(3.6)}
 is obtained by the duality.
For $u\in C^\infty(S^m\tau')$, we have
$$
(\delta iu)_{i_1\dots i_mj}=
g^{kl}[\sigma(i_1\dots i_mjk)(g_{jk}
{\nabla}_{\!l}u_{i_1\dots i_m})]
$$
 by the definition of $i$ and~$\delta$,
 and by the equality ${\nabla}g=0$.
The expression in brackets represents the sum
 over~$\Pi_{m+2}$.
 We divide all the terms of the sum into four groups as follows:
 The first group includes the products of the form
 $g_{rs}\nabla_lu_{t_1\dots t_m}$ for $\{r,s\}=\{j,k\}$;
 the second group (the third group)
 consists of the products such that
 $k\in\{r,s\}$ and $j\notin\{r,s\}$ ($k\notin\{r,s\}$ and $j\in\{r,s\}$),
   and the fourth group contains all the remaining terms. We thus obtain
$$
\begin{aligned}
\frac {(m+1)(m+2)} {2}
(\delta iu)_{i_1\dots i_mj}=
g^{kl}\Big[&g_{jk}{\nabla}_{\!l}u_{i_1\dots i_m}
+
\sum\limits_{a=1}^{m}g_{ki_a}{\nabla}_{\!l}
u_{ji_1\dots\widehat{i_a}\dots i_m}\\
&+
\sum\limits_{a=1}^{m}g_{ji_a}{\nabla}_{\!l}
u_{ki_1\dots\widehat{i_a}\dots i_m}
+
\sum\limits_{1\leq a<b\leq m}g_{i_ai_b}{\nabla}_{\!l}
u_{jki_1\dots\widehat{i_a}\dots\widehat{i_b}\dots i_m}
\Big],
\end{aligned}
$$
where $\wedge$ over an index means the index is omitted.
 Using the identity $ g^{kl}g_{jk}=\delta^l_k$,
 where $(\delta^l_k)$ is the Kronecker tensor,
 we rewrite the last formula in the form
$$
\begin{aligned}
\frac {(m+1)(m+2)} {2}
(\delta iu)_{i_1\dots i_mj}&=
{\nabla}_{\!j}u_{i_1\dots i_m}
+
\sum\limits_{a=1}^{m}{\nabla}_{\!i_a}
u_{ji_1\dots\widehat{i_a}\dots i_m}\\
&+
\sum\limits_{a=1}^{m}g_{ji_a}
(\delta u)_{i_1\dots\widehat{i_a}\dots i_m}
+
\sum\limits_{1\leq a<b\leq m}g_{i_ai_b}
(\delta u)_{ji_1\dots\widehat{i_a}\dots\widehat{i_b}\dots i_m}.
\end{aligned}
$$
The sum of the first two summands on the right-hand side of this formula equals\\
 $(m+1)(du)_{ji_1\dots i_m}$,
 and the sum of the last two summands equals
 $\frac{m(m+1)}2(i\delta u)_{ji_1\dots i_m}$.
\end{proof}

\begin{Lemma}\label{3.4}
The following equalities hold on $C^\infty(S^m\tau')$\dv
\begin{equation}
dq=qd-\frac {m} {n+2m-2}i\delta p,\quad
q\delta=\delta q-\frac {m-1} {n+2m-4} pdj,
            \label{(3.7)}
\end{equation}
\begin{equation}
dp=pd+\frac {m} {n+2m-2} i\delta p,\quad
p\delta=\delta p+\frac {m-1} {n+2m-4} pdj.
            \label{(3.8)}
\end{equation}
\end{Lemma}

\begin{proof}
 As above, the formulas are written in dual pairs
 and \eqref{(3.8)} is obtained from~\eqref{(3.7)} by the substitution $q=E-p$.
 Hence it suffices to prove the first of
 formulas~\eqref{(3.7)}.

By \eqref{(3.1)} and \eqref{(3.6)},
$$
jid=jdi=\left(\frac {2} {m+3}\delta+
\frac {m+1} {m+3}dj\right)i.
$$
The left and right multiplications of this formula by $(ji)^{-1}$
 yield
$$
d(ji)^{-1}=\frac {m+1} {m+3}(ji)^{-1}d+
\frac {2} {m+3}(ji)^{-1}\delta i(ji)^{-1}.
$$
In virtue of \eqref{(2.15)} and \eqref{(3.6)}, this gives
$$
\begin{aligned}
dq=di(ji)^{-1}j&=id(ji)^{-1}j=
i\Big(\frac {m-1} {m+1}(ji)^{-1}d+
\frac {2} {m+1}(ji)^{-1}\delta i(ji)^{-1}\Big)j\\
&=
\frac {m-1} {m+1}i(ji)^{-1}\Big(
\frac {m+1} {m-1}jd-\frac {2} {m-1}\delta\Big)
+\frac {2} {m+1}i(ji)^{-1}\delta q\\
&=
qd-\frac {2} {m+1}i(ji)^{-1}\delta(E-q),
\end{aligned}
$$
i.e.,
\begin{equation}
dq=qd-\frac {2} {m+1}i(ji)^{-1}\delta p.
            \label{(3.9)}
\end{equation}

The equality $jp=0$ follows from the definition of~$p$.
 According to~\eqref{(2.11)} and~\eqref{(3.1)},
 this implies
$$
(ji)\delta p=
\frac {2(n+2m-2)} {m(m+1)}\delta p+
\frac {(m-1)(m-2)} {m(m+1)}ij\delta p=
\frac {2(n+2m-2)} {m(m+1)}\delta p.
$$
Hence the equality
\begin{equation}
(ji)^{-1}\delta p=
\frac {m(m+1)} {2(n+2m-2)}\delta p
            \label{(3.10)}
\end{equation}
holds on $C^\infty(S^m\tau')$. Substitution of~\eqref{(3.10)}
 into~\eqref{(3.9)} implies the first
 formula of~\eqref{(3.7)}.
\end{proof}

\bigskip

Let us now demonstrate how does Theorem~\ref{1.1} imply Corollary~\ref{1.2}.
 Let~$u$ and~$v$ satisfy the hypotheses of Corollary~\ref{1.2}, and
 let  $\tilde u=pu$. Then
\begin{equation}
j\tilde u=0.
            \label{(3.11)}
\end{equation}
Applying the operator~$p$ to~\eqref{(1.2)}, we obtain $pdu=0$.
 Transformation of the left-hand side of this equality
 by~\eqref{(3.8)} implies
$$
\left(d-\frac {m} {n+2m-2}i\delta\right)\tilde u=0.
$$
We denote $\tilde v=\frac{m}{n+2m-2}\delta\tilde u$ and rewrite the last formula as
\begin{equation}
d\tilde u=i\tilde v.
            \label{(3.12)}
\end{equation}
As is seen from~\eqref{(2.15)}, the operator $p=E-q$ is
 represented in coordinates by a matrix whose elements are
 rational functions of the components~$g_{ij}$ of the metric
 tensor. Therefore conditions~\eqref{(1.4)} imply
 the similar conditions for $\tilde u=pu$:
\begin{equation}
\tilde u(x_0)=0,\quad{\nabla}\tilde u(x_0)=0,\quad\dots,\quad
{\nabla}^l\tilde u(x_0)=0.
            \label{(3.13)}
\end{equation}
According to \eqref{(3.11)}--\eqref{(3.13)}, ~$\tilde u$
 satisfies the hypotheses of Theorem~\ref{1.1}. Assuming the theorem to be
 valid, we obtain $\tilde u=pu=0$. This means the existence of
 $w$ such that~$u=iw$. Theorem~\ref{1.3} implies
 Corollary~\ref{1.4} in a similar way.

\section{Proof of Theorem \ref{1.3}}

According to Theorem~\ref{1.1} whose proof will be given below,
 Theorem~\ref{1.3} follows from a weaker statement formulated in

\begin{Lemma} \label{4.}
Let $\Gamma$~be a smooth hypersurface in a Riemannian manifold~$M$.
 If tensor fields $u\in C^\infty(S^m\tau')$ and
 $v\in C^\infty(S^{m-1}\tau')$
 satisfy the conditions
 $$
 du=iv,
 \quad ju=0,
 \quad u\bigr\vert_\Gamma=0
 $$
 then $u$ and~$v$ vanish on~$\Gamma$ together with all their derivatives.
\end{Lemma}

\begin{proof}
The statement is trivial for $m=0$.
  Let $m\ge 1$. We prove by induction in~$k$
  the validity of the following statement:
$$
\begin{aligned}
&u|_\Gamma=0,\ {\nabla}u|_\Gamma=0,\ \dots, {\nabla}^ku|_\Gamma=0,\\
&v|_\Gamma=0,\ {\nabla}v|_\Gamma=0,\ \dots, {\nabla}^{k-1}v|_\Gamma=0.
\end{aligned}
$$
For $k=0$, the statement coincides with the hypothesis  $u\bigr\vert_\Gamma=0$.
 Assume the required statement to be true for some $k\ge 0$.

We choose a coordinate system
 $(x^1,\dots,x^n)=(x'{}^1,\dots,x'{}^{n-1},y)$ in some neighborhood of $x_0\in\Gamma$  so
 that~$\Gamma$ is defined by the equation $y=0$ and
  $g_{in}=\delta_{in}$.
 Here and below  $(\delta_{ij})$~is the Kronecker tensor.
By the induction hypothesis,
\begin{equation}
 \begin{aligned}
 \frac{\partial^l}{\partial y^l}
 \partial^\beta_{x'}u_{i_1\dots i_m}\biggr\vert_{y=0}=0
 &\ \ \mbox{for}\ \ l\leq k,\\
 \noalign{\vskip7pt}
 \frac{\partial^l}{\partial y^l}
 \partial^\beta_{x'}v_{i_1\dots i_{m-1}}\biggr\vert_{y=0}=0
 &\ \ \mbox{for}\ \ l\leq k-1
 \end{aligned}
 \label{(4.1)}
 \end{equation}
  for an arbitrary $(n-1)$-variate index $\beta$.

The equality $du=iv$ has the following form in the chosen coordinates:
  \begin{align*}
 &\nabla_{i_1}u_{i_2\dots i_{m+1}}+
 \nabla_{i_2}u_{i_1i_3\dots i_{m+1}}+\dots+
 \nabla_{i_{m+1}}u_{i_1\dots i_m}\\
 &\qquad\qquad=(m+1)\sigma(i_1\dots i_{m+1})
 \big(g_{i_1i_2}v_{i_3\dots i_{m+1}}
 \big).
 \end{align*}
  Applying the operator
 $\frac{\partial^k}{\partial y^k}\bigr\vert_{y=0}$
  to this equality and taking~\eqref{(4.1)} into account, we obtain
 \begin{align}
 &\frac{\partial^k}{\partial y^k}
 \bigg(\frac{\partial u_{i_2\dots i_{m+1}}}
            {\partial x^{i_1}}+\dots+
       \frac{\partial u_{i_1\dots i_{m}}}
            {\partial x^{i_{m+1}}}
 \bigg)\biggr\vert_{y=0}\nonumber\\
 &\qquad =(m+1)\sigma(i_1\dots i_{m+1})
 \bigg(g_{i_1i_2}\frac{\partial^k v_{i_3\dots i_{m+1}}}
                      {\partial y^k}\biggr\vert_{y=0}
 \bigg).\label{(4.2)}
 \end{align}

Hereafter we use the following agreement:
  Greek indices vary from~1 to~$n-1$, and the summation from 1 to~$n-1$ is assumed over repeated Greek indices.
Set $(i_1,\dots,i_{m+1})=(\alpha_1,\dots,\alpha_{m+1})$
  in~\eqref{(4.2)}. Then the left-hand side of~\eqref{(4.2)} equals zero by~\eqref{(4.1)}
  and we obtain
\begin{equation}
\sigma(\alpha_1\dots \alpha_{m+1})
\left(g_{\alpha_1\alpha_2}\left.
\frac {\partial^k v_{\alpha_3\dots \alpha_{m+1}}}{\partial y^k}
\right|_{y=0}\right)=0.
            \label{(4.3)}
\end{equation}

We rewrite \eqref{(4.2)} in the form
 \begin{align}
 &\frac{\partial^k}{\partial y^k}
 \bigg(\frac{\partial u_{i_2\dots i_{m+1}}}
           {\partial x^{i_1}}+\dots+
      \frac{\partial u_{i_1\dots i_{m}}}
           {\partial x^{i_{m+1}}}
 \bigg)\biggr\vert_{y=0}\nonumber\\
 &\qquad=\frac1{m!}
        \sum_{\pi\in\Pi_{m+1}}g_{i_{\pi(1)}i_{\pi(2)}}
        \frac{\partial^k v_{i_{\pi(3)}\dots i_{\pi(m+1)}}}
             {\partial y^k}\biggr\vert_{y=0}.
 \label{(4.4)}
 \end{align}
Let $0\le s\le m$. We set
  $(i_1,\dots,i_{m-s})=(\alpha_1,\dots,\alpha_{m-s})$ and
 $i_{m-s+1}=\dots=i_{m+1}=n$ in~\eqref{(4.4)}.
 By~\eqref{(4.1)}, the first~$m-s$ summands
 on the left-hand side of~\eqref{(4.4)} are equal to zero and
 the last~$s+1$ summands coincide, i.e.,
\begin{equation}
\frac {\partial^k} {\partial y^k}
\left.\left(
\frac {\partial u_{i_2\dots i_{m+1}}} {\partial x^{i_1}}
+\dots+
\frac {\partial u_{i_1\dots i_{m}}} {\partial x^{i_{m+1}}}
\right)\right|_{y=0}
=
(s+1)
\left.\frac {\partial^{k+1}
u_{\alpha_1\dots\alpha_{m-s}n\dots n}} {\partial y^{k+1}}
\right|_{y=0}.
            \label{(4.5)}
\end{equation}
Let us analyze the right-hand side of~\eqref{(4.4)} for the chosen indices.
 If $\pi(1)\leq m-s$ and $\pi(2)>m-s$ then $g_{i_{\pi(1)}i_{\pi(2)}}=0$.
 Similarly,  $g_{i_{\pi(1)}i_{\pi(2)}}=0$ if  $\pi(1)>m-s$ and $\pi(2)\leq m-s$. Therefore
 \begin{align}
 &\hskip-5mm
 \sum_{\pi\in\Pi_{m+1}}g_{i_{\pi(1)}i_{\pi(2)}}
 \frac{\partial^k v_{i_{\pi(3)}\dots i_{\pi(m+1)}}}
     {\partial y^k}\biggr\vert_{y=0}\nonumber\\
 &\quad=\hskip-5mm
 \sum_{\pi\in\Pi_{m+1}(s)}
 \frac{\partial^k v_{i_{\pi(3)}\dots i_{\pi(m+1)}}}
     {\partial y^k}\biggr\vert_{y=0}+
 \hskip-2mm
 \sum_{\pi\in\Pi'_{m+1}(s)}g_{i_{\pi(1)}i_{\pi(2)}}
 \frac{\partial^k v_{i_{\pi(3)}\dots i_{\pi(m+1)}}}
     {\partial y^k}\biggr\vert_{y=0},
 \label{(4.6)}
 \end{align}
 where
  \begin{align*}
 \Pi_{m+1}(s) &=\big\{\pi\in\Pi_{m+1}\mid\pi(1)>m-s,\ \pi(2)>m-s
                \big\},\\
 \Pi'_{m+1}(s)&=\big\{\pi\in\Pi_{m+1}\mid\pi(1)\leq m-s,\ \pi(2)\leq m-s
               \big\}.
 \end{align*}
All the summands of the first sum on the right-hand side
 of~\eqref{(4.6)} coincide because~$v$ is symmetric.
 And the total amount of the summands is
$ (m-1)!s(s+1) $, i.e.,
\begin{equation}
\sum\limits_{\pi\in\Pi_{m+1}(s)}
\left.
\frac {\partial^k v_{i_{\pi(3)}\dots i_{\pi(m+1)}}
} {\partial y^k}
\right|_{y=0}
=
(m-1)!s(s+1)
\left.
\frac {\partial^k
v_{\alpha_1\dots\alpha_{m-s}n\dots n}}{\partial y^k}
\right|_{y=0}.
            \label{(4.7)}
\end{equation}
For $s=0$, the right-hand side of~\eqref{(4.7)} is equal to zero due
  to the factor~$s$.

The second sum on the right-hand side of~\eqref{(4.6)} is
 obviously equal to
\begin{equation}
c(m,s)\sigma(\alpha_1\dots\alpha_{m-s})
\left(g_{\alpha_1\alpha_2}
\left.
\frac {\partial^k
v_{\alpha_3\dots\alpha_{m-s}n\dots n}}{\partial y^k}
\right|_{y=0}\right),
            \label{(4.8)}
\end{equation}
where $c(m,s)=(m-1)!(m-s)(m-s-1)$~is total amount of elements
 in~$\Pi'_{m+1}(s)$.
Substitute~\eqref{(4.7)} and \eqref{(4.8)} into~\eqref{(4.6)} to obtain
  \begin{align}
 &\hskip-5mm\frac1{(m-1)!}
 \sum_{\pi\in\Pi_{m+1}}g_{i_{\pi(1)}i_{\pi(2)}}
 \frac{\partial^k v_{i_{\pi(3)}\dots i_{\pi(m+1)}}}
     {\partial y^k}\biggr\vert_{y=0} 
 =s(s+1)
 \frac{\partial^kv_{\alpha_1\dots\alpha_{m-s}n\dots n}}
     {\partial y^k}\biggr\vert_{y=0}\nonumber\\
 &\hskip-5mm\quad\qquad+(m-s)(m-s-1)\sigma(\alpha_1\dots\alpha_{m-s})
 \Bigg(g_{\alpha_1\alpha_2}
      \frac{\partial^kv_{\alpha_3\dots\alpha_{m-s}n\dots n}}
           {\partial y^k}\biggr\vert_{y=0}
 \Bigg).
 \label{(4.9)}
 \end{align}
Next, we substitute \eqref{(4.5)} and~\eqref{(4.9)} into~\eqref{(4.4)}
 \begin{align}
 &\hskip-5mm\frac{\partial^{k+1}u_{\alpha_1\dots\alpha_{m-s}n\dots n}}
                {\partial y^{k+1}}\biggr\vert_{y=0} 
 =\frac{s}{m}
 \frac{\partial^kv_{\alpha_1\dots\alpha_{m-s}n\dots n}}
      {\partial y^{k}}\biggr\vert_{y=0}\nonumber\\
 &\hskip-5mm\quad\qquad+
  \frac{(m-s)(m-s-1)}
       {m(s+1)}\sigma(\alpha_1\dots\alpha_{m-s})
 \Bigg(g_{\alpha_1\alpha_2}
      \frac{\partial^kv_{\alpha_3\dots\alpha_{m-s}n\dots n}}
           {\partial y^k}\biggr\vert_{y=0}
 \Bigg).
 \label{(4.10)}
 \end{align}

 We define the tensor fields
 $$
 z^{(s)}\in C^\infty(S^{m-s}\tau'_\Gamma)\ \ (0\le s\le m),
 \quad
 w^{(s)}\in C^\infty(S^{m-s}\tau'_\Gamma)\ \ (1\le s\le m)
 $$
 on $\Gamma$ as follows:
\begin{equation}
z^{(s)}_{\alpha_1\dots\alpha_{m-s}}=
\left.\frac {\partial^{k+1}u_{\alpha_1\dots\alpha_{m-s}n\dots n}}
{\partial y^{k+1}}\right|_{y=0},\quad
w^{(s)}_{\alpha_1\dots\alpha_{m-s}}=
\left.\frac {\partial^{k}v_{\alpha_1\dots\alpha_{m-s}n\dots n}}
{\partial y^{k}}\right|_{y=0}.
            \label{(4.11)}
\end{equation}
For convenience, we also define
 $w^{(0)}=0$, $w^{(m+1)}=0$, and $w^{(m+2)}=0$. Then~\eqref{(4.10)}
 can be written in the coordinate-free form
\begin{equation}
z^{(s)}=\frac {s} {m}w^{(s)}+
\frac {(m-s)(m-s-1)} {m(s+1)}iw^{(s+2)}
\quad (0\leq s\leq m),
            \label{(4.12)}
\end{equation}
and~\eqref{(4.3)} can be written as $iw^{(1)}=0$.
 Since $i$~ is a monomorphism, this implies
\begin{equation}
w^{(1)}=0.
            \label{(4.13)}
\end{equation}

If $m=1$ then $v$~is a scalar function and
 \eqref{(4.13)} gives
 $\frac{\partial^kv}{\partial y^{k}}\bigr\vert_{y=0}=0$.
 Equation~\eqref{(4.12)} implies
$$
z_\alpha^{(0)}=
\left.\frac {\partial^{k+1}u_\alpha}
{\partial y^{k+1}}\right|_{y=0}=0,\quad
 \left.\frac {\partial^{k+1}u_n}
{\partial y^{k+1}}\right|_{y=0}=z^{(1)}=w^{(1)}=0.
$$
This justifies the induction step in the case of~$m=1$.
 Therefore, we assume $m\ge 2$ in the rest of the proof.

In the chosen coordinates, the equation $ju=0$ is written as follows:
$$
u_{i_1\dots i_{m-2}nn}+g^{\beta\gamma}
u_{\beta\gamma i_1\dots i_{m-2}}=0.
$$
Differentiating this identity $k+1$ times with respect to $y$, we obtain
$$
\left.\frac {\partial^{k+1}u_{i_1\dots i_{m-2}nn}}
{\partial y^{k+1}}\right|_{y=0}+
g^{\beta\gamma}\left.
\frac {\partial^{k+1}u_{\beta\gamma i_1\dots i_{m-2}}}
{\partial y^{k+1}}\right|_{y=0}=0.
$$
We set $ (i_1,\dots,i_{m-s})=(\alpha_1,\dots,\alpha_{m-s}) $ and
$ i_{m-s+1}=\dots i_{m-2}=n $ in the last formula to obtain
$$
z^{(s)}+jz^{(s-2)}=0\quad (2\leq s\leq m).
$$
This implies
\begin{equation}
z^{(2s)}=(-j)^sz^{(0)}\quad \mbox{for}\quad 0\leq 2s\leq m,
\quad
z^{(2s+1)}=(-j)^sz^{(1)}\quad\mbox{for}\quad
0\leq 2s+1\leq m.
                                           \label{(4.14)}
\end{equation}

 Setting $s=m$ and then  $s=m-1$ in~\eqref{(4.12)}, we get
$$
w^{(m)}=z^{(m)},\quad w^{(m-1)}=\frac {m} {m-1}z^{(m-1)}.
$$
 Taking~\eqref{(4.14)} into account, this implies
\begin{equation}
 \begin{aligned}
 &w^{(2m')}=(-j)^{m'}z^{(0)}
   &&\mbox{for}\ \ m=2m',\\
 &w^{(2m')}=\frac{2m'+1}{2m'}(-j)^{m'}z^{(0)}
   &&\mbox{for}\ \ m=2m'+1,
 \end{aligned}
 \label{(4.15)}
 \end{equation}
 \vskip-7pt
 \begin{equation}
 \begin{aligned}
 &w^{(2m'-1)}=\frac{2m'}{2m'-1}(-j)^{m'-1}z^{(1)}
   &&\mbox{for}\ \ m=2m',\\
 &w^{(2m'+1)}=(-j)^{m'}z^{(1)}
   &&\mbox{for}\ \ m=2m'+1.
 \end{aligned}
 \label{(4.16)}
 \end{equation}

Now, we are going to prove the representations
 \begin{alignat}2
 w^{(2s)}
 &=(-1)^s
 \sum_{l=0}^{[m/2]-s}
 a(m,s,l)i^lj^{s+l}z^{(0)}
  &&\ \ \mbox{for}\ \ 0<2s\leq m,
 \label{(4.17)}\\
  w^{(2s+1)}
 &=(-1)^s\sum_{l=0}^{[m-1/2]-s}
 b(m,s,l)i^lj^{s+l}z^{(1)}
  &&\ \ \mbox{for}\ \ 0<2s+1\leq m,
 \label{(4.18)}
 \end{alignat}
 with some
 {\it positive\/}
 coefficients~$a(m,s,l)$ and~$b(m,s,l)$,
 where~$[\boldsymbol{\cdot}]$ denotes, as usual,
 the integer part of a number.
 For $0\le m-2s\le 1$, formula~\eqref{(4.17)}
coincides with~\eqref{(4.15)}. We shall prove~\eqref{(4.17)}
   by induction in~$m-2s$.
   Let~$s>0$ and $m-2s\ge 2$. If we take $s:=2s$ in~\eqref{(4.12)} then we get
$$
z^{(2s)}=\frac {2s} {m}w^{(2s)}+
\frac {(m-2s)(m-2s-1)} {m(2s+1)}iw^{(2s+2)}.
$$
We express $ w^{(2s)} $ from this equation
$$
w^{(2s)}=\frac {m} {2s}z^{(2s)}-
\frac {(m-2s)(m-2s-1)} {2s(2s+1)}iw^{(2s+2)}.
$$
We replace the first term on the right-hand side by its value \eqref{(4.14)} and replace the second term by its expression from the inductive hypothesis
\eqref{(4.17)}
$$
w^{(2s)}=(-1)^s\Big(\frac {m} {2s}j^sz^{(0)}+
\frac {(m-2s)(m-2s-1)} {2s(2s+1)}
\sum\limits_{l=0}^{\left[\frac {m} {2}\right]-s-1}
a(m,s+1,l)i^{l+1}j^{s+l+1}z^{(0)}\Big).
$$
Changing the summation index, we transform this expression to the form
$$
w^{(2s)}=(-1)^s\Big(\frac {m} {2s}j^sz^{(0)}+
\frac {(m-2s)(m-2s-1)} {2s(2s+1)}
\sum\limits_{l=1}^{\left[\frac {m} {2}\right]-s}
a(m,s+1,l-1)i^{l}j^{s+l}z^{(0)}\Big).
$$
This is equivalent to~\eqref{(4.17)} with
$$
a(m,s,l)=\left\{
\begin{array}{l}
m/2s\quad \mbox{for}\quad l=0,\\
[0.3cm]
\frac {(m-2s)(m-2s-1)} {2s(2s+1)}a(m,s+1,l-1)
\quad \mbox{for}\quad l\geq 1.
\end{array}\right.
$$
Thus,  representation~\eqref{(4.17)} is proved.
 The proof of~\eqref{(4.18)} is quite similar.

Set $ s=0 $ in \eqref{(4.12)}
\begin{equation}
z^{(0)}=(m-1)iw^{(2)}.
            \label{(4.19)}
\end{equation}
By \eqref{(4.17)},
$$
w^{(2)}=-
\sum\limits_{l=0}^{\left[\frac {m} {2}\right]-1}
a(m,1,l)i^{l}j^{l+1}z^{(0)}.
$$
Substitution of the last expression into~\eqref{(4.19)} gives
\begin{equation}
\Big[E+
(m-1)\sum\limits_{l=1}^{[m/2]}
a(m,1,l-1)i^{l}j^{l}\Big]z^{(0)}=0,
            \label{(4.20)}
\end{equation}
where $E$~is the identity operator.
The operator in the  brackets is nondegenerate  since the coefficients of the sum are positive
  and the operator $i^lj^l$~is nonnegative.
  Hence,~\eqref{(4.20)} implies $z^{(0)}=0$.
  So, according to~\eqref{(4.14)} and~\eqref{(4.17)},
  $z^{(2s)}=0$ and $w^{(2s)}=0$ for all~$s$.

By \eqref{(4.13)}, $w^{(1)}=0$.
 On the other hand, setting~$s=0$ in~(4.18), we see
$$
w^{(1)}=
\Big[\sum\limits_{l=0}^{\left[\frac {m-1} {2}\right]}
b(m,0,l)i^{l}j^{l}\Big]z^{(1)}=0.
$$
Since the operator in the brackets is nondegenerate,
  $z^{(1)}=0$. Together with~\eqref{(4.14)} and~\eqref{(4.18)}, this gives
  $z^{(2s+1)}=0$ and $w^{(2s+1)}=0$ for all~$s$.

We have proved $z^{(s)}=0$ and $w^{(s)}=0$ for
  all~$s$.  Recalling definition~\eqref{(4.11)}, we see
$$
\left.\frac {\partial^{k+1}u_{i_1\dots i_m}}
{\partial y^{k+1}}\right|_{y=0}=0,\quad
\left.\frac {\partial^{k}v_{i_1\dots i_m}}
{\partial y^{k}}\right|_{y=0}=0,
$$
and this is the finish of the inductive step.
\end{proof}

\section{Proof of Theorem \ref{1.5}}

We start with the following observation: the tensor fields $\lambda$ and $\tilde f$ can be eliminated from
  the system~\eqref{(1.5)}--\eqref{(1.7)}. Indeed, let
  $f,\tilde f\in H^k(S^m\tau')$, $v\in H^{k+1}(S^{m-1}\tau')$, and
  $\lambda\in H^k(S^{m-2}\tau')$ satisfy~\eqref{(1.5)}--\eqref{(1.7)}.
  Applying the operator~$j$ to~\eqref{(1.5)}, we get
$$
jf=jdv+ji\lambda.
$$
Express $ \lambda $ from this
\begin{equation}
\lambda=(ji)^{-1}j(f-dv),
            \label{(5.1)}
\end{equation}
and substitute the result into \eqref{(1.5)}
$$
f=dv+i(ji)^{-1}j(f-dv)+\tilde f.
$$
Due to~\eqref{(2.15)}, this equality can be written
 in the form
$$
f=dv+q(f-dv)+\tilde f
$$
or
$$
(E-q)f=(E-q)dv+\tilde f,
$$
where $ E $ is the identity operator. Since $ E-q=p $,
\begin{equation}
pf=pdv+\tilde f.
            \label{(5.2)}
\end{equation}
To eliminate $ \tilde f $, we apply the operator $ \delta $ to equation \eqref{(5.2)}. Taking $ \delta \tilde f=0 $ into account, we obtain
$$
\delta pf=\delta pdv.
$$
Hence, $v$ is a solution to the boundary value problem
\begin{equation}
(\delta pd)v=h,\quad v|_{\partial M}=0
            \label{(5.3)}
\end{equation}
with
\begin{equation}
h=\delta pf\in H^{k-1}(S^{m-1}\tau').
            \label{(5.4)}
\end{equation}

Recall that the subbundle $\Ker j$ of the vector
 bundle~$S^*\tau'$ was defined in~Section~2. The right-hand side~$h$
 of equation~\eqref{(5.3)} belongs to $H^{k-1}(\Ker j)$ by~\eqref{(5.4)} and~\eqref{(3.1)}.
The desired solution~$v$ to problem~\eqref{(5.3)} must be a
 section of~$\Ker j$ since the requirement $jv=0$ is contained in~\eqref{(1.6)}. Finally, $\delta pd$ can be
 considered as a differential operator on the vector
 bundle~$\Ker j$, i.e.,
$$
\delta pd:C^\infty(\mbox{\rm Ker}\,j)\rightarrow C^\infty(\mbox{\rm Ker}\,j),
$$
since $q(\delta pd)=(q\delta p)d=0$ in view of~\eqref{(3.4)}.
  So,~\eqref{(5.3)} can be considered as a boundary value problem
  on the bundle~$\Ker j$.
 We shall prove this is an elliptic problem with  zero kernel and co-kernel.
 Then, applying the theorem on regular solvability
 of elliptic problems,
 we shall deduce that, for every $h\in H^k(\Ker j)$ $(k\ge 0)$, problem~\eqref{(5.3)} has a unique solution
 $v\in H^{k+2}(\Ker j)$ satisfying the stability estimate
$$
\|v\|_{H^{k+2}}\leq C\|h\|_{H^k}.
$$
Setting $h=\delta pf$ and defining~$\lambda$ and~$\tilde f$ by formulas~\eqref{(5.1)}
 and~\eqref{(5.2)}, we get \eqref{(1.5)}--\eqref{(1.8)}.
Theorem~\ref{1.5} is thus reduced to the following proposition:

\begin{Theorem} \label{5.1}
Let $M$ be a compact connected Riemannian manifold with nonempty boundary.
Being considered  on the vector bundle~$\Ker j$, the boundary value problem~\eqref{(5.3)} is
elliptic and has zero kernel and co-kernel.
\end{Theorem}

 \begin{proof}
 We start
with checking ellipticity
 of the operator~$\delta pd$ on~$\Ker j$.
 The principal symbols~$\sigma_1(d)$ and $\sigma_1(\delta)$ of the
 operators~$d$ and~$\delta$ at a point~$(x,\xi)\in T'$ are
$$
\sigma_1(d)=\sqrt{-1}i_\xi\quad \sigma_1(\delta)=
\sqrt{-1}j_\xi;
$$
here $\sqrt{-1}$~is the imaginary unit. Hence,
$$
\sigma_2(\delta pd)=-j_\xi pi_\xi.
$$
For $x\in M$, let $\Ker^m_xj=\{f\in S^mT'_x\mid jf=0\}$.
  We have to prove the operator
\begin{equation}
j_\xi pi_\xi:\mbox{\rm Ker}^m_xj\rightarrow\mbox{\rm Ker}^m_xj
            \label{(5.5)}
\end{equation}
is an isomorphism for every $m\ge 0$ and every $0\ne\xi\in T'_x$.

The operator $j_\xi pi_\xi$ is easily seen to be
 nonnegative. Indeed,
$$
\langle j_\xi pi_\xi f,f\rangle =
\langle pi_\xi f,i_\xi f\rangle =
\langle pi_\xi f,pi_\xi f\rangle =
|pi_\xi f|^2.
$$
Therefore verification of the ellipticity of~$\delta pd$
  reduces to the following proposition.

\begin{Lemma} \label{5.2}
If a tensor $f\in S^mT'_x$ satisfies the conditions
 $$
 jf=0,
 \quad pi_\xi f=0
 $$
 for some $0\neq\xi\in T'_x$
 then $f=0$.
\end{Lemma}

To prove Lemma \ref{5.2}, we need the following:
\begin{Lemma} \label{5.3}
The commutation formula
$$
pi_\xi=i_\xi p-\frac {2} {m+1}i(ji)^{-1}j_\xi p
$$
holds on $ S^m\tau' $.
\end{Lemma}

\begin{proof} The commutation formula
\begin{equation}
ji_\xi=\frac {2} {m+1}j_\xi+\frac {m-1} {m+1}i_\xi j
\quad \mbox{on}\quad S^m\tau'
            \label{(5.6)}
\end{equation}
is checked by direct calculations in coordinates, and we omit
 them. Using~\eqref{(2.15)} and~\eqref{(5.6)}, we obtain
$$
qi_\xi=i(ji)^{-1}ji_\xi=i(ji)^{-1}(ji_\xi)=
i(ji)^{-1}\left(\frac {2} {m+1}j_\xi+
\frac {m-1} {m+1}i_\xi j\right).
$$
Hence,
\begin{equation}
qi_\xi=
i(ji)^{-1}\left(\frac {2} {m+1}j_\xi+
\frac {m-1} {m+1}i_\xi j\right)
\quad \mbox{on}\quad S^m\tau'.
            \label{(5.7)}
\end{equation}

Using \eqref{(5.6)} again, we get
$$
jii_\xi=(ji_\xi)i=
\left(\frac {2} {m+3}j_\xi+\frac {m+1} {m+3}i_\xi j\right)i=
\frac {2} {m+3}j_\xi i+\frac {m+1} {m+3}i_\xi (ji).
$$
 Multiplying the extreme parts of this formula by~$(ji)^{-1}$
  from the left and from the right, we obtain
$$
i_\xi(ji)^{-1}
=
\frac {2} {m+3}(ji)^{-1}j_\xi i(ji)^{-1}+
\frac {m+1} {m+3}(ji)^{-1}i_\xi.
$$
Hence,
\begin{equation}
(ji)^{-1}i_\xi=
\frac {m+3} {m+1}i_\xi(ji)^{-1}-
\frac {2} {m+1}(ji)^{-1}j_\xi i(ji)^{-1}
\quad \mbox{on}\quad S^m\tau'.
            \label{(5.8)}
\end{equation}

We transform the second summand on the right-hand side
  of~\eqref{(5.7)} with the htlp of~\eqref{(5.8)}. The summand equals zero in the case of $m=0$ and of
  $m=1$ due to the factor~$j$ on its right. Therefore we assume
  $m\ge 2$. Since the
  operator~$j$ acts before $(ji)^{-1}i_\xi$,
  the value~$m$ in~\eqref{(5.8)} should be changed to~$m-2$.
  Thus, the result of the transformation is as follows:
$$
qi_\xi=\frac {2} {m+1}i(ji)^{-1}j_\xi+\frac {m-1} {m+1}
i\left(\frac {m+1} {m-1}i_\xi(ji)^{-1}-
\frac {2} {m-1}(ji)^{-1}j_\xi i(ji)^{-1}\right)j.
$$
We arrange this formula as
$$
qi_\xi=i_\xi i(ji)^{-1}j+
\frac {2} {m+1}i(ji)^{-1}j_\xi(E-i(ji)^{-1}j)
$$
and use \eqref{(2.15)} again to obtain
\begin{equation}
qi_\xi=i_\xi q+\frac {2} {m+1}i(ji)^{-1}j_\xi p
\quad \mbox{on}\quad S^m\tau'.
            \label{(5.9)}
\end{equation}
We have thus proved~\eqref{(5.9)} in the case of $m\ge 2$.
 Actually, ~\eqref{(5.9)} is valid for any $m\ge 0$. Indeed,
 both sides of this formula are equal to zero in the case of $m=0$. In the case of $m=1$,
 \eqref{(5.9)} has the form
$$
qi_\xi=\frac {1} {n}ij_\xi
\quad \mbox{on}\quad \tau'
$$
and can be easily checked.

Substituting $q=E-p$ into~\eqref{(5.9)},
we complete the proof of  Lemma~\ref{5.3}.
\end{proof}

\bigskip

\begin{proof}[Proof of Lemma~$\ref{5.2}$]
  The statement of the lemma is trivial for $m=0$ since $pi_\xi f=i_\xi f$ in the latter case and~$i_\xi$ is a monomorphism
 for $\xi\neq 0$. So we assume $m\ge 1$.

 Let $f\in S^mT'_x$ satisfy the equalities $jf=0$ and $pi_\xi f=0$.
 By Lemma~\ref{5.3}, the second equality implies
$$
0=pi_\xi f=i_\xi pf-\frac {2} {m+1}i(ji)^{-1}j_\xi pf.
$$
Since $ pf=f $, this equality is simplified to the following one:
$$
i_\xi f-\frac {2} {m+1}i(ji)^{-1}j_\xi f=0.
$$
Taking the scalar product of this with $ i_\xi f $, we get
$$
\langle i_\xi f,i_\xi f\rangle
-\frac {2} {m+1}\langle i(ji)^{-1}j_\xi f,i_\xi f\rangle =0
$$
or
\begin{equation}
\langle j_\xi i_\xi f,f\rangle
-\frac {2} {m+1}\langle (ji)^{-1}j_\xi f,ji_\xi f\rangle =0.
            \label{(5.10)}
\end{equation}

The operators $i_\xi$ and $j_\xi$ satisfy the commutation formula
\begin{equation}
j_\xi i_\xi f=\frac {|\xi|^2} {m+1}f+
\frac {m} {m+1}i_\xi j_\xi f
\quad \mbox{for}\quad f\in S^m\tau'
            \label{(5.11)}
\end{equation}
(see~\cite[Lemma 3.3.3]{[mb]}).

 Since $ jf=0$, formula~\eqref{(5.6)} implies
\begin{equation}
ji_\xi f=\frac {2} {m+1}j_\xi f.
            \label{(5.12)}
\end{equation}

Using \eqref{(2.11)} and taking $ jf=0 $ into account, we deduce
$$
(ji)(j_\xi f)=\frac {2(n+2m-2)} {m(m+1)}j_\xi f.
$$
Applying the operator~$(ji)^{-1}$ to this equation, we infer
\begin{equation}
(ji)^{-1}(j_\xi f)=\frac {m(m+1)}{2(n+2m-2)} j_\xi f.
            \label{(5.13)}
\end{equation}

Substitute \eqref{(5.11)}--\eqref{(5.13)}
 into~\eqref{(5.10)} to obtain
$$
|\xi|^2|f|^2+m\left(1-\frac {2} {n+2m-2)}\right)
|j_\xi f|^2=0.
$$
Both coefficients of the equality are positive in the case of $ n\ge 2$, $m\ge 1 $,
 and $\xi\neq 0$. Hence,~$ f=0$.
 \end{proof}

\bigskip

We have thus proved the ellipticity of the
 principal symbol
 $j_\xi pi_\xi$ of the operator $-\delta pd $ on
 the bundle~$\Ker j$. Actually, we have shown
 the principal symbol is positive.
 This implies the ellipticity of the boundary value
 problem~\eqref{(5.3)}. Indeed, as is known \cite[Chapter 5, Proposition 11.10]{[Ta]}, the positivity of
 the principal symbol implies
 the Lopatinski{\u\i} condition for the Dirichlet problem.

Next, we are going to prove the triviality of the kernel of the boundary value problem~\eqref{(5.3)}. Let $ v\in H^k(\Ker j)$ $(k\ge 2)$ be a solution
 to the homogeneous problem
\begin{equation}
(\delta pd)v=0,\quad v|_{\partial M}=0.
            \label{(5.14)}
\end{equation}
Due to the ellipticity, $v$ is smooth: $v\in C^\infty(\Ker j)$.
 Applying Green's formula from Theorem~\ref{3.1}, we have
$$
(pdv,pdv)_{L^2}=(pdv,dv)_{L^2}=-(\delta pdv,v)_{L^2}=0,
$$
i.e., $ pdv=0$. Hence, $v$~is a trace-free
 conformal Killing
 field. According to Theorem~\ref{1.3}, if such a field satisfies
 the homogeneous boundary condition $v\bigr\vert_{\partial M}=0$ then it is identically zero.

 Let $S^m\tau'\bigr\vert_{\partial M}$ denote the restriction of the
 bundle~$S^m\tau'$ to the boundary.
 To prove the triviality of the co-kernel of the boundary value problem~\eqref{(5.3)}, we need the following proposition.

\begin{Lemma}  \label{5.4}
If a tensor field
 $u\in C^\infty\big(S^m\tau'\bigr\vert_{\partial M}\big)$
 satisfies the condition $ju=0$ then there exists
 $v\in C^\infty(S^m\tau')$ satisfying the conditions
 $v\bigr\vert_{\partial M}=0$$,$ $jv=0$ and such that
\begin{equation}
j_\nu pdv|_{\partial M}=u,
            \label{(5.15)}
\end{equation}
where $\nu$~is
 the outward
 normal vector to the boundary.
\end{Lemma}

The proof of Lemma~\ref{5.4} will be given below. We now finish the proof of Theorem~\ref{5.1} with the help of the lemma.

Assume a field $w\in C^\infty(\Ker j)$ to be orthogonal
 to the range of the operator of the boundary value problem~\eqref{(5.3)},
 i.e.,
\begin{equation}
(w,\delta pdv)_{L^2}=0
            \label{(5.16)}
\end{equation}
for every $v\in C^\infty(\Ker j)$ satisfying the boundary condition $v\bigr\vert_{\partial M}=0$.
 We have to show  $w\equiv 0$.
We first choose~$v$ such that $\supp v\subset M\setminus\partial M$.
 Green's formula and~\eqref{(5.16)} imply
$$
(\delta pdw,v)_{L^2}=(w,\delta pdv)_{L^2}=0.
$$
Since $v\in C^\infty_0(\Ker j)$ is arbitrary, this means
\begin{equation}
\delta pdw=0.
            \label{(5.17)}
\end{equation}

For an arbitrary
 $u\in C^\infty\big(\Ker j\bigr\vert_{\partial M}\big)$,
 Lemma~\ref{5.4} guaranties the existence of
 some $v\in C^\infty(\Ker j)$
 which satisfies~\eqref{(5.15)} and vanishes on the boundary.
 With the help of Green's formula, \eqref{(5.15)}--\eqref{(5.17)} yield
$$
0=(w,\delta pdv)_{L^2}=
(\delta pdw,v)_{L^2}+
\int\limits_{\partial M}\langle w,j_\nu pdv\rangle dV'=
\int\limits_{\partial M}\langle w,u\rangle dV'.
$$
This means  $w\bigr\vert_{\partial M}=0$ since $u$ is arbitrary.
 So,~$w$ belongs to the kernel of the boundary problem operator.
 As was already proved, such~$w$ must be identically equal to zero.
 Theorem~\ref{5.1} is proved.
 \end{proof}

\bigskip

\begin{proof}[Proof of Lemma~$\ref{5.4}$]
 In order to simplify the notation, we give here the proof
 only  in the case of an odd~$m$.
  The case of an even~$m$ is considered in a similar way.
  Both the cases can be considered simultaneously but with much more complicated notation.

In virtue of the condition $v\bigr\vert_{\partial M}=0$, \eqref{(5.15)} can be considered as an algebraic equation in the unknown $ \partial v/\partial\nu\bigr\vert_{\partial M}$.
 We are going to prove the existence and uniqueness of a solution to the equation under the condition $ ju=\nobreak0$. Moreover, the solution satisfies
 $j\partial v/\partial\nu\bigr\vert_{\partial M}=0$ as will be shown.
  Then the proof of the existence is realized by choosing a section~$v$ of the
 vector bundle~$\Ker j$ with prescribed boundary values $v\bigr\vert_{\partial M}=0$ and
 $\partial v/\partial\nu\bigr\vert_{\partial M}$.

We choose
 normal boundary coordinates
 $(x^1,\dots,x^n)=(x^1,\allowbreak\dots,x^{n-1},y) $
 in a neighborhood of a boundary point
 so that $g_{in}=\delta_{in}$ and the boundary is defined by the
 equation $y=0$.
 Below the Greek indices change from~1 to~$n-1$.
 We define the tensor fields
 $u^{(s)},v^{(s)}\in C^\infty(S^s\tau'_{\partial M})$
 for $0\le s\le 2m+1$ by the formulas
\begin{equation}
v^{(s)}_{\alpha_1\dots\alpha_s}=
\left.\frac {\partial v_{\alpha_1\dots\alpha_sn\dots n}}
{\partial y}\right|_{y=0},\quad
u^{(s)}_{\alpha_1\dots\alpha_s}=2(m+1)(n+4m)
u_{\alpha_1\dots\alpha_sn\dots n} .
            \label{(5.18)}
\end{equation}
The condition $ju=0$ is expressed in terms of~$u^{(s)}$ as follows:
$$
u^{(s)}+ju^{(s+2)}=0.
$$
From this,
\begin{equation}
u^{(2s)}=(-j)^{m-s}u^{(2m)},\quad
u^{(2s+1)}=(-j)^{m-s}u^{(2m+1)}.
            \label{(5.19)}
\end{equation}

 Using \eqref{(3.8)}, we transform~\eqref{(5.15)} to the form
$$
j_\nu\left.\left(dpv-\frac {2m+1} {n+4m}i\delta pv\right)
\right|_{\partial M}=u.
$$
If $ jv=0$ then $pv=v$ and the equation simplifies to the following one:
\begin{equation}
j_\nu\left.\left(dv-\frac {2m+1} {n+4m}i\delta v\right)
\right|_{\partial M}=u.
            \label{(5.20)}
\end{equation}
We are going to derive some recurrent formulas from \eqref{(5.20)}
 which  uniquely determine the tensors~$v^{(s)}$.

In the normal boundary coordinates,
 the vector~$\nu$ has the coordinates $(0,\dots,0,1)$
 and equation~\eqref{(5.20)} takes the form
\begin{equation}
(dv)_{ni_1\dots i_{2m+1}}|_{y=0}-
\frac {2m+1} {n+4m}
(i\delta v)_{ni_1\dots i_{2m+1}}|_{y=0}=
u_{i_1\dots i_{2m+1}}.
            \label{(5.21)}
\end{equation}

Set $(i_1,\dots,i_s)=(\alpha_1,\dots,\alpha_s)$ and
 $i_{s+1}=\dots=i_{2m+2}=n$
 in the equality
$$
(dv)_{i_1\dots i_{2m+2}}=
\frac {1} {2m+2}\left(
\frac {\partial v_{i_2\dots i_{2m+2}}} {\partial x^{i_1}}+
\frac {\partial v_{i_1i_3\dots i_{2m+2}}} {\partial x^{i_2}}+
\dots+
\frac {\partial v_{i_1\dots i_{2m+2}}} {\partial x^{i_{2m+1}}}
\right).
$$
The first $s $ summands on the right-hand side vanish on~$\partial M$
 since $ v\bigr\vert_{\partial M}=0 $. The last~$2m-s+2$ summands
 are pairwise equal. Hence,
\begin{equation}
(dv)_{\alpha_1\dots\alpha_s n\dots n}|_{y=0}=
\frac {2m-s+2} {2(m+1)}
v^{(s)}_{\alpha_1\dots\alpha_s}.
            \label{(5.22)}
\end{equation}
Similarly, we deduce
\begin{equation}
(\delta v)_{\alpha_1\dots\alpha_s n\dots n}|_{y=0}=
v^{(s)}_{\alpha_1\dots\alpha_s}.
            \label{(5.23)}
\end{equation}

Setting $ (i_1,\dots,i_s)=(\alpha_1,\dots,\alpha_s) $ and
$ i_{s+1}=\dots=i_{2m+2}=n $ in the equality
$$
(i\delta v)_{i_1\dots i_{2m+2}}=\frac {1} {(2m+2)!}
\sum\limits_{\pi\in\Pi_{2m+2}}
g_{i_{\pi(1)}i_{\pi(2)}}
(\delta v)_{i_{\pi(3)}\dots i_{\pi(2m+2)}}
$$
and analyzing the right-hand side in the same way as has been used for deriving
\eqref{(4.9)}, we obtain
\begin{align*}
 (i\delta v)_{\alpha_1\dots\alpha_s n\dots n}\bigr\vert_{y=0}
 &=\frac{(2m-s+1)(2m-s+2)}
        {(2m+1)(2m+2)}(\delta v)_{\alpha_1\dots\alpha_s n\dots n}
       \biggr\vert_{y=0}\\
 &\quad+\frac{s(s-1)}
             {(2m+1)(2m+2)}\sigma(\alpha_1\dots\alpha_s)
 \big(g_{\alpha_1\alpha_2}
 (\delta v)_{\alpha_3\dots\alpha_s n\dots n}\bigr\vert_{y=0}
 \big).
 \end{align*}
With the help of \eqref{(5.23)}, this gives
\begin{align}
 (i\delta v)_{\alpha_1\dots\alpha_s n\dots n}\bigr\vert_{y=0}
 &=\frac{(2m-s+1)(2m-s+2)}
        {(2m+1)(2m+2)}v^{(s)}_{\alpha_1\dots\alpha_s}\nonumber\\
 &\qquad+\frac{s(s-1)}{(2m+1)(2m+2)}
 \big(iv^{(s-2)}
 \big)_{\alpha_1\dots\alpha_s}.
 \label{(5.24)}
 \end{align}

We set $(i_1,\dots,i_s)=(\alpha_1,\dots,\alpha_s)$ and
 $i_{s+1}=\dots=i_{2m+2}=n$ in~\eqref{(5.21)}.
Then we substitute values~\eqref{(5.22)} and~\eqref{(5.24)} for the summands on the left-hand side of~\eqref{(5.21)}
 and value~\eqref{(5.18)} for the right-hand side of~\eqref{(5.21)}. In such the way we obtain the recurrent formula
$$
(2m-s+2)(n+2m+s-1)v^{(s)}-s(s-1)iv^{(s-2)}=u^{(s)}.
$$
In view of \eqref{(5.19)}, this formula can be rewritten as
 \begin{align}
 2(m{-}s{+}1)(n{+}2m{+}2s{-}1)v^{(2s)}{-}2s(2s{-}1)iv^{(2s{-}2)}
 &=(-j)^{m{-}s}u^{(2m)},
 \label{(5.25)}\\
 (2m{-}2s{+}1)(n{+}2m{+}2s)v^{(2s{+}1)}{-}2s(2s{+}1)iv^{(2s{-}1)}
 &=(-j)^{m{-}s}u^{(2m{+}1)}.
 \label{(5.26)}
 \end{align}

Formulas~\eqref{(5.25)} and \eqref{(5.26)} imply the following
 representations:
 \begin{align}
 v^{(2s)}  &=\sum_{k=0}^{s}a(s,k)i^kj^{m-s+k}u^{(2m)},
 \label{(5.27)}\\
 v^{(2s+1)}&=\sum_{k=0}^{s}b(s,k)i^kj^{m-s+k}u^{(2m+1)},
 \label{(5.28)}
 \end{align}
 with some coefficients~$a(s,k)$ and~$b(s,k)$ which depend
 on~$n,m,s$, and~$k$ only.
 The dependence on~$n$ and~$m$ is not
 indicated explicitly
 since the values of these two parameters
 are fixed in the proof.

 Substitution of~\eqref{(5.27)} and~\eqref{(5.28)}
  into~\eqref{(5.25)} and~\eqref{(5.26)}, respectively,
  imply the following recurrent relations:
 \begin{equation}
  \begin{aligned}
 a(s,0)&=\dfrac{(-1)^{m-s}}
              {2(m-s+1)(n+2m+2s-1)},\\
 a(s,k)&=\dfrac{(2s-1)}
              {(m-s+1)(n\+2m\+2s\_1)}a(s\_1,k\_1)
 \ \ \mbox{for}\ \ 1\leq k\leq s,
 \end{aligned}
 \label{(5.29)}
  \end{equation}
 \begin{equation}
  \begin{aligned}
  b(s,0)&=\dfrac{(-1)^{m-s}}
              {(2m-2s+1)(n+2m+2s)},\\
 b(s,k)&=\dfrac{2s(2s+1)}
              {(2m-2s+1)(n+2m+2s)}b(s-1,k-1)
 \ \ \mbox{for}\ \ 1\le k\le s.
 \end{aligned}
 \label{(5.30)}
  \end{equation}
The coefficients $a(s,k)$ and~$b(s,k)$ are uniquely determined
 by equations (5.29) and~(5.30), and
 formulas~\eqref{(5.27)} and~\eqref{(5.28)} show that
 the tensors~$v^{(2s)}$ and~$v^{(2s+1)}$ are uniquely
 determined by~$u^{(2m)}$ and~$u^{(2m+1)}$.

Finally, we have to prove the tensors $v^{(2s)}$
 and~$v^{(2s+1)}$  satisfy the equations
  \begin{align}
 v^{(2s)}+jv^{(2s+2)}&=0
   \ \ \mbox{for}\ \ 0\le s\le m-1,
 \label{(5.31)}\\
 v^{(2s+1)}+jv^{(2s+3)}&=0
   \ \ \mbox{for}\ \ 0\le s\le m-1
 \label{(5.32)}
  \end{align}
  that are equivalent to the relation
  $j\frac{\partial v}{\partial \nu}\big\vert_{\partial M}=0$
  in view of the first formula in~\eqref{(5.18)}.

 Applying the operator~$ j $ to equation~\eqref{(5.27)},
 we obtain
\begin{equation}
jv^{(2s+2)}=\sum\limits_{k=0}^{s+1}a(s+1,k)
ji^kj^{m-s+k-1}u^{(2m)}.
            \label{(5.33)}
\end{equation}
We transpose the factors $j$ and $i^k$ on the right-hand side
 of~\eqref{(5.33)} with the help of Lemma~\ref{2.2}. We have to set
 $n:=n-1$ and $m:=2s-2k+2$ in the statement of the ltmma since $j^{m-s+k-1}u^{(2m)}$
 is the tensor
 of rank $2s-2k+2$ on the $(n-1)$-dimensional manifold $\partial M$.
 So we have
 \begin{align*}
 jv^{(2s+2)}= \sum_{k=0}^{s+1}a(s+1,k)
 &\bigg(\frac{k(n+4s-2k+1)}
            {(s+1)(2s+1)}i^{k-1}j^{m-s+k-1}\\
 &\qquad\qquad+
      \frac{(s-k+1)(2s-2k+1)}
           {(s+1)(2s+1)}i^{k}j^{m-s+k}
 \bigg)u^{(2m)}.
 \end{align*}
 This equality can be transformed as follows:
 \begin{align}
 \hskip-5mm jv^{(2s+2)}= \sum_{k=0}^s
 &\bigg(\frac{(s\_k\+1)(2s\_2k\+1)}
            {(s\+1)(2s\+1)}a(s\+1,k)\nonumber\\
 &\quad+
 \frac{(k\+1)(n\+4s\_2k\_1)}
     {(s\+1)(2s\+1)}a(s\+1,k\+1)
 \bigg)i^{k}j^{m\_s\+k}u^{(2m)}.
 \label{(5.34)}
 \end{align}
Substitution of~\eqref{(5.27)} and~\eqref{(5.34)}
 into~(5.31) gives
\begin{align*}
 \sum_{k=0}^s
 \bigg[a(s,k)&+\frac{(s-k+1)(2s-2k+1)}
                  {(s+1)(2s+1)}a(s+1,k)\\
            &+ \frac{(k+1)(n+4s-2k-1)}
           {(s+1)(2s+1)}a(s+1,k+1)
 \bigg]i^{k}j^{m-s+k}u^{(2m)}=0.
 \end{align*}
 Since $u^{(2m)}$~is an arbitrary tensor, the expression in the
 brackets must be equal to zero for all~$s$ and~$k$, i.e.,
 \begin{equation}
 \begin{aligned}
 a(s,k)&+\frac{(s-k+1)(2s-2k+1)}
             {(s+1)(2s+1)}a(s+1,k)\\
      &+
 \frac{(k+1)(n+4s-2k-1)}
     {(s+1)(2s+1)}a(s+1,k+1)=0\ \
 \mbox{for}\ \ 0\le k\le s\le m-1.
 \end{aligned}
 \label{(5.35)}
 \end{equation}
Similarly, (5.32) is equivalent to the equation
 \begin{equation}
 \begin{aligned}
 b(s,k)&+
 \frac{(s-k+1)(2s-2k+3)}
     {(s+1)(2s+3)}b(s+1,k)\\
     &+
 \frac{(k+1)(n+4s-2k+1)}
      {(s+1)(2s+1)}b(s+1,k+1)=0\ \
\mbox{for}\ \ 0\le k\le s\le m-1.
 \end{aligned}
 \label{(5.36)}
  \end{equation}

 We have to prove the following statement: Being defined by recurrent formulas~(5.29) and~(5.30), the coefficients~$a(s,k)$ and~$b(s,k)$
  satisfy equations~\eqref{(5.35)} and~\eqref{(5.36)}, respectively.
  This can be proved with the help of the following explicit formulas for the
  coefficients:
  \begin{align}
 a(s,k)
 &= \frac{(-1)^{m\_s\_k}}2
 \frac{s!(2s\_1)!!(m\_s)!}
      {(n\+2m\+2s\_1)!!}\nonumber\\
 &\qquad\qquad \qquad \times
 \frac{(n\+2m\+2s\_2k\_3)!!}
      {(s\_k)!(m\_s\+k\+1)!(2s\_2k\_1)!!},
 \label{(5.37)}\\\noalign{\vskip12pt}
 b(s,k)
 &=(-1)^{m\_s\_k}
 \frac{s!(2s\+1)!!(2m\_2s\_1)!!}
      {(n\+2m\+2s)!!}\nonumber\\
 &\qquad\qquad \qquad \times
 \frac{2^k(n\+2m\+2s\_2k\_2)!!}
      {(s\_k)!(2m\_2s\+2k\+1)!!(2s\_2k\+1)!!}.
 \label{(5.38)}
 \end{align}
  Here we use the standard notation:
 $$
 (2k)!!=2^kk!,
 \ \ (2k+1)!!=(2k+1)(2k-1)\dots 1, 
 \ \ (-1)!!=1.
 $$
 Formulas~\eqref{(5.37)} and~\eqref{(5.38)}
 are proved by substituting them into
 recurrent formulas~(5.29) and~(5.30) and checking the validity of the resulting equations.
 Then the validity of \eqref{(5.35)} and~\eqref{(5.36)} is proved by substitution of values~\eqref{(5.37)} and~\eqref{(5.38)} followed by
 a direct calculation.
\end{proof}

\section{Proof of Theorems 1.6 and 1.7}

 In this section, for a Riemannian manifold, we use the notions of a semibasic tensor field and the vertical and horizontal derivatives of
 such a field.
 The corresponding definitions are presented in~\cite[Ch.\,3]{[mb]}
 (see also~\cite[\S\,4]{[DP1]} where the case of a Finsler manifold
 is considered as well).
  We denote the space of smooth
 semibasic $(r,s)$-tensor fields on~$TM$ by $C^\infty(\beta^r_sM)$,
 and~$\nablav$, $\nablah:C^\infty(\beta^r_sM)\rightarrow C^\infty(\beta^r_{s+1}M)$~
 denote the vertical and horizontal derivatives, respectively.

 \begin{proof}[Proof of Theorem \rm 1.6]
  Let $u\in C^\infty(S^{m}\tau'_M)$~be a trace-free
 conformal Killing
 tensor field, i.e., $ ju=0$ and $ du=iv $ for some
 $ v\in C^\infty(S^{m-1}\tau'_M)$. We assume here $m\ge1$
 since the statement of the theorem is trivial in the case of $ m=0 $.
 Define the function
 $U\in C^\infty(TM)=C^\infty(\beta^0_0M)$ as follows:
 \vskip-0.0007pt\noindent
 \begin{equation}\label{U-def}
 U(x,\xi)=u_{i_1\dots i_m}(x)\xi^{i_1}\cdots\xi^{i_m}.
 \end{equation}
 The function is homogeneous with respect to $\xi$,
  \begin{equation}\label{U-hom}
   U(x,t\xi)=t^m U(x,\xi),
 \end{equation}
  and satisfies the kinetic equation
  \begin{equation}\label{HU=v}
 HU(x,\xi)=|\xi|^2v_{i_1\dots i_{m-1}}(x)
 \xi^{i_1}\cdots\xi^{i_{m-1}},
 \end{equation}
  where $H$~denotes differentiation
 along the geodesic flow,
 and $HU(x,\xi)=\xi^i\nablah_{i}U$.
 Since $ju=0$, the function~$U$ satisfies the equation
  \begin{equation}\label{DU=0}
 \Deltav U=0,
 \end{equation}
 \vskip-7pt\noindent
 where $\Deltav=\nablav{}^i\nablav_{i}$~is the vertical Laplacian
  (see Lemma~2.4).

 Let us derive the commutation formula for $\Deltav$ and~$H$.
 Since the vertical and horizontal derivatives commute,
 $$
 \Deltav HU= \nablav{}^i\nablav_i\big(\xi^j\nablah_{j}U
                     \big)= \nablav{}^i\big(\nablah_iU+\xi^j\nablah_j\nablav_iU
            \big)= 2\nablav{}^i\nablah_iU+H\Deltav U.
 $$
 By~\eqref{DU=0}, this gives
  \begin{equation}\label{DH}
 \Deltav HU=2\nablav{}^i\nablah_iU.
 \end{equation}

 We write the Pestov identity for the function~$U$
 (see~\cite{[mb]})
  \begin{equation}\label{Pes}
   2\big\langle\kern-.3mm\nablah U,\nablav HU
  \big\rangle-\nablav_i
  \big(\nablah{}^iU\cdot HU
  \big)=\big\vert\nablah U
        \big\vert^2+\nablah_iw^i-R_{\xi}
  \big(\nablav U
  \big),
 \end{equation}
  where
  \begin{equation}\label{RU}
 R_\xi\big(\nablav U
      \big)=R_{ijkl}\xi^i\xi^k\nablav{}^j U\cdot\nablav{}^lU
 \end{equation}
  and $w$~is some semibasic vector field on~$TM$. It depends on $ U $
 quadratically
 but its value is not relevant now.
  Since the sectional curvature is nonpositive, we have
  \begin{equation}\label{RU<0}
 R_\xi\big(\nablav U
      \big)\le 0.
 \end{equation}

 We transform the first summand on the left-hand side of
 equation~\eqref{Pes} with the help of~\eqref{DH} as follows:
 \begin{align*}
 2\big\langle\kern-.3mm\nablah U,\nablav HU
  \big\rangle
 &=2\nablah{}^iU\cdot\nablav_i(HU)\\
 &=\nablav_i\big(2\nablah{}^iU\cdot HU
            \big)-2\nablav{}^i\nablah_iU\cdot HU\\
 &=-\Deltav HU\cdot HU+\nablav_i\big(2\nablah{}^iU\cdot HU
                                \big)\\
 &=-\nablav_i\nablav{}^iHU\cdot HU+\nablav_i
    \big(2\nablah{}^i U\cdot HU
    \big)\\
 &=-\nablav_i\big(\nablav{}^i HU\cdot HU
             \big)+\nablav{}^iHU\cdot\nablav_iHU+
                   \nablav_i\big(2\nablah{}^i U\cdot HU
                            \big)\\
 &=\big\vert\nablav HU
   \big\vert^2+\nablav_i
   \big(2\nablah{}^iU\cdot HU-\nablav{}^iHU\cdot HU
   \big).
 \end{align*}
Substitute this value into~\eqref{Pes}
 $$
 \big\vert\nablav HU
 \big\vert^2+\nablav_i
 \big(\nablah{}^iU\cdot HU-\nablav{}^iHU\cdot HU
 \big)=\big\vert\nablah U
       \big\vert^2+\nablah_iw^i-R_\xi
 \big(\nablav U
 \big).
 $$
We integrate this equality over~$\Omega M$ versus to
   the Liouville volume form~$d\Sigma$ and transform the integrals of divergent terms by the
   Gauss-Ostrogradsky formulas
       (see~\cite[Theorem 3.6.3]{[mb]})
 \begin{equation}
 \int_{\Omega M}
 \Big(\big\vert\nablav HU
      \big\vert^2+(n+2m)
      \big\langle\xi,\nablah U-
                     \nablav HU
      \big\rangle HU
 \Big)\,d\Sigma
 =\int_{\Omega M}
 \Big(\big\vert\nablah U
      \big\vert^2-R_\xi
      \big(\nablav U
      \big)
 \Big)\,d\Sigma.
 \label{IP}
 \end{equation}
   The coefficient $(n+2m)$ appears here because the
 semibasic
 vector field $\big(\nablah U-\nablav HU\big)HU $
 is homogeneous of degree~$2m+1$ in $\xi$.
 With the help of the Euler formula for homogeneous functions
 \begin{equation}
 \big\langle\xi,\nablav HU
 \big\rangle=(m+1)HU
 \label{EF}
 \end{equation}
 and of $\big\langle\xi,\nablah U\big\rangle=HU$, formula (\ref{IP}) takes the form
 \begin{equation}\label{IPes}
 \int_{\Omega M}
 \Big(\big\vert\nablav HU
      \big\vert^2-m(n+2m)|HU|^2
 \Big)\,d\Sigma= \int_{\Omega M}
 \Big(\big\vert\nablah U
      \big\vert^2-R_\xi
      \big(\nablav U
      \big)
 \Big)\,d\Sigma.
 \end{equation}

 Now, we estimate the left-hand side of~\eqref{IPes} as follows.
At an
 arbitrary point $(x,\xi)\in\Omega M$, we represent the vector~$\nablav HU$  in the form
  \begin{equation}\label{botHU}
 \nablav HU=\lambda\xi+\nablav{}^\perp HU,
 \quad
 \big\langle\xi,\nablav{}^\perp HU
 \big\rangle=0.
 \end{equation}
  Here $\lambda=\lambda(x,\xi)$ is some scalar function.
 The second summand of the representation has a clear geometrical
 sense: If $\psi_x=HU\bigr\vert_{\Omega_x M}$ is a
 restriction of~$HU$ to the unit sphere~$\Omega_x M $ then
 $\nablav{}^\perp HU(x,\xi)\,{=}\,\nabla\psi_x(\xi)$ is the gradient
 of the function~$\psi_x$ at the point $\xi\in\Omega_x M$.
 Formula~\eqref{HU=v} implies that
 $\psi_x(\xi)=v_{i_1\dots v_{i_{m-1}}}(x)\xi^{i_1}\cdots\xi^{i_{m-1}}$.
 Applying the Euler formula~\eqref{EF},
 we see that $\lambda=(m+1)HU$. Thus, \eqref{botHU} implies
 \begin{equation}\label{11}
 \big\vert\nablav HU
 \big\vert^2=(m+1)^2|HU|^2+|\nabla\psi_x|^2.
 \end{equation}
 By Green's formula,
 $$
 \int_{\Omega_x M}|\nabla\psi_x|^2\,d\omega(\xi)=-
 \int_{\Omega_x M}\psi_x\Delta_\omega\psi_x\,d\omega(\xi),
 $$
 where $\Delta_\omega$~is the spherical Laplacian on~$\Omega_x M$.
 The eigenvalues of  $-\Delta_\omega$ are
 $\lambda_k=k(n+k-2)$, $k=0,1,\dots$, and the spherical harmonics
 of degree~$k $ are the eigenfunctions corresponding
 to~$\lambda_k$. Since $\psi_x$~ is a polynomial of degree~$m-1$,
 the last integral can be estimated
 as follows:
  \begin{align*}
 \int_{\Omega_x M}|\nabla\psi_x|^2\,d\omega(\xi)
 &=-\int_{\Omega_x M}\psi_x\Delta_\omega\psi_x\,d\omega(\xi)\\
 &\le\sup_{k\le m-1}\lambda_k
     \int_{\Omega_x M}|\psi_x|^2\,d\omega(\xi)\\
 &=(m-1)(n+m-3)\int_{\Omega_x M}|HU|^2\, d\omega(\xi).
 \end{align*}
  Together with~\eqref{11}, this imply
 $$
 \int_{\Omega M}
 \Big(\big\vert\nablav HU
      \big\vert^2-m(n+2m)|HU|^2
 \Big)\,d\Sigma\le-(2m+n-4)
 \int_{\Omega M}|HU|^2\, d\Sigma.
 $$
  Taking this inequality into account, we derive from~\eqref{IPes}
 \begin{equation}\label{12}
 -(2m+n-4)
 \int_{\Omega M}|HU|^2\,d\Sigma
 \ge\int_{\Omega M}
 \Big(\big\vert\nablah U
      \big\vert^2-R_\xi
      \big(\nablav U
      \big)
 \Big)\,d\Sigma.
 \end{equation}

 The coefficient $(2m+n-4) $ is nonnegative since $m\ge 1$ and $n\ge 2$.
 Hence the left-hand side of~\eqref{12} is nonpositive.
 At the same time, the right-hand side of~\eqref{12}
 is nonnegative in virtue of~\eqref{RU<0}. Thus, both sides of~\eqref{12} are equal to zero.
 In particular,
  $\big\vert\nablah U\big\vert^2-R_\xi  \big(\nablav U\big)=0$
 on~$\Omega M$.
 Applying~\eqref{RU<0} once more, we obtain
  \begin{equation}\label{RU=0}
 R_\xi(\nablav U)=0
 \end{equation}
 and $\nablah U=0$ on~$\Omega M$. Hence, $\nablah U$
 is identically zero on~$TM$.
 Now,~\eqref{U-def} implies
 $ 0=\nablah_i U=\nabla_iu_{i_1\dots i_m}\xi^{i_1\dots i_m}$.
 Thus, $\nabla u$ is identically zero on~$M$,
 i.e., $u$ is
 absolutely parallel.

 Now, we prove the statement: $u(x_0)=0$ if all sectional curvatures at
 the point~$x_0$ are negative. This implies~$u\equiv 0$
 since~$u$ is absolutely parallel.

 Given $\xi\in\Omega_{x_0}$,  like in \eqref{botHU}, we represent the vector $\nablav U(x_0,\xi)$ in the form
 \begin{equation}\label{14}
 \nablav U(x_0,\xi)=\mu(\xi)\xi+\nablav{}^\perp U(x_0,\xi), \quad
 \big\langle\xi,\nablav{}^\perp U(x_0,\xi) \big\rangle=0.
 \end{equation}
 Here~$\mu$ is some scalar function.
  We claim $\nablav{}^\perp U(x_0,\xi)=\nobreak0 $ for
  all $\xi\in\Omega_{x_0}$. Indeed, assume
  $\nablav{}^\perp U(x_0,\xi)\ne 0 $ for some~$\xi$.
 Then, substituting \eqref{14} into~\eqref{RU} and using symmetries
 of the curvature tensor, we infer
 $$
 R_\xi\big(\nablav{}^\perp U(x_0,\xi)
     \big)=K
     \big(x_0,\xi\wedge\nablav U(x_0,\xi)
     \big)\big\vert\nablav{}^\perp U
          \big\vert^2<0.
 $$
 Here $K\big(x_0,\xi\wedge\nablav U(x_0,\xi)\big)$~is
 the value of the sectional curvature at the point~$x_0$ in the two-dimensional direction
  $\xi\wedge\nablav U(x_0,\xi)$.
 The last inequality contradicts~\eqref{RU=0}.

 Hence, $\nablav{}^\perp U(x_0,\xi)=0$ for all
 $\xi\in\Omega_{x_0}$. This means that
 \begin{equation}\label{15}
 U\bigr\vert_{\Omega_{x_0}M}=c=\const.
 \end{equation}
  In the case of an odd $m$, the constant~$c$ must be equal to zero since the function $U(x,\xi) $ is odd in~$\xi$.
  In the case of $m=2l>0$, \eqref{U-def} and~\eqref{15} imply $u(x_0)=cg^l$.
  The condition $ ju=0 $ implies $c=0$. Thus, $u(x_0)=0$
  for all~$m$.

 We have proved the statements of Theorem 1.6 concerning
 trace-free conformal Killing tensor fields.  We now prove
 the statements of the theorem concerning a Killing field by induction in the rank~$m$ of the field.

  The statements are valid in the cases of $m=0$ and of $m=1$ since
  a Killing vector field is a trace-free
  conformal Killing field
  as well. Assume $m\ge2$ and let $u\in C^\infty(S^m\tau_M)$~be a Killing
  tensor field
  of rank~$m$.
  Represent~$u$ in the form
 \begin{equation}\label{16}
 u=\tilde u+iv,
 \end{equation}
 where $\tilde u$ satisfies the condition $j\tilde u=0$.
 So, $\tilde u$~is a trace-free
 conformal Killing field, and thus,
 $\nabla\tilde u=0$. Applying the operator~$d$ to~\eqref{16}, we
 obtain $ i dv=0$. Hence, $ dv=0$, i.e., $v$~is a Killing field. We
 obtain $\nabla v=0$ by the inductive assumption. So, both summands
 on the right-hand side of~\eqref{16} are
 absolutely parallel, and  $u$ is also an absolutely parallel field.

 The remaining statement on Killing fields is proved
 in a similar way.
  \end{proof}

 \begin{proof}[Proof of Theorem \rm 1.7]
 We now need the following corollary of Theorem A
 from~\cite{[DP2]} (see also the remark after
 formulation of the theorem in~\cite{[DP2]}).

 \begin{proposition}\label{cohomag}
 Let $(M,g)$~be a closed Riemannian manifold without
 conjugate points.
 Let $h\in C^\infty(M)$ and $\theta\in C^\infty(\tau'_M)$.
 If the equation
 $$
 HU(x,\xi)=h(x)+\theta_i(x)\xi^i,
 \quad (x,\xi)\in\Omega M,
 $$
 has a solution $U\in C^\infty(\Omega M)$ then $h=0$ and
 $\theta$~is an exact
 $1$-form.
 \end{proposition}

 Assume now  $u$ to be a conformal Killing
 covector field, i.e.,
 \begin{equation}\label{du=v}
 du=iv
 \end{equation}
  for some function~$v$ on~$M$.
  Define the function $U\in C^\infty(\Omega M)$ by
 $$
 U(x,\xi)=u_i(x)\xi^i,
 \quad(x,\xi)\in\Omega M.
 $$
 As follows from~\eqref{du=v}, $U$ satisfies
 the kinetic equation
 $$
 HU(x,\xi)=v(x)\ \ \text{on}\ \ \Omega M.
 $$
 Applying Proposition~\ref{cohomag}, we obtain $v\equiv 0$.
 Since
 $$
 \big(du(x)
 \big)_{ij}\xi^i\xi^j=HU(x,\xi)=v(x)=0,
 $$
  we see that $du=0$, i.e., $u$ is a Killing covector field.
  If the geodesic flow of~$(M,g)$ has a dense orbit in~$\Omega M$
  then $HU=0$ implies $U\equiv\const$.
  This means that $u\equiv 0$.

 Assume now $u$ to be a trace-free
 conformal Killing symmetric field
 of rank~$2$, i.e.,
 $$
 du=iv,
 \quad ju=0
 $$
  for some covector field~$v$. Define the function
 $$
 U(x,\xi)=u_{ij}(x)\xi^i\xi^j
 $$
  on~$\Omega M$. It satisfies the equation
 \begin{equation}\label{HU=vxi}
 HU(x,\xi)=v_i(x)\xi^i,
 \quad (x,\xi)\in\Omega M.
 \end{equation}
  By Proposition \ref{cohomag}, $v$ is an exact 1-form, i.e.,
  \begin{equation}\label{v=dphi}
 v=d\varphi
 \end{equation}
  for some function~$\varphi$ on~$M$.
  Formulas \eqref{HU=vxi} and~\eqref{v=dphi}
  imply
 \begin{equation}\label{u-phig}
 d(u-\varphi g)=0.
 \end{equation}
 Since $ju=0$, this means $u$ is the trace-free part of the
 Killing field $u-\varphi g$. On the other hand,
 the trace-free part of a Killing tensor field is obviousely a trace-free
 conformal Killing field.

 If the geodesic flow~$(M,g)$ has a dense in~$\Omega M$ orbit
 then $u-\varphi g=cg$ for some
 constant~$c$, as follows from~\eqref{u-phig}. Together with the condition $ju=0$,
 this means~$u=0$.
 Theorem 1.6 is proved.
 \end{proof}

 \section{Comparison of differential operators
              modulo low order terms}

 In our studying differential operators on
 tensor fields, we will usually ignore low order terms.
 In order to simplify the exposition, we introduce the following notation:
 If~$A$ and~$B$~are two differential expressions, we write
 $$
 Au=Bu
 \pmod{\nabla^ku}\ \ \mbox{on}\ \ {\mathcal A}
 $$
 if there exists a differential operator~$L^{(k)}$ of order~$k$
 such that
 $$
 Au=Bu+L^{(k)}u
 $$
  for all tensor fields~$u$ belonging to the subspace~$\mathcal A$ of the
  space $C^\infty(\otimes^*\tau')$. The choice of the subspace
  will be mostly clear from the context.
  Similar notation is used for differential operators
  depending on several variables.

 \begin{Lemma}\label{3.5}
 For $k\ge 1$ and $u\in C^\infty(\otimes^m\tau')$$,$
 $$
 {\nabla}_{\!j_1\dots j_k}u_{i_1\dots i_m}= \sigma(j_1\dots j_k)
 \big(\nabla_{\!j_1\dots j_k}u_{i_1\dots i_m}
 \big)
 \pmod{{\nabla}^{k-2}u}.
 $$
 \end{Lemma}

 We omit the proof that can be easily carried out by induction
 in $ k $ starting from the following formula for the second order
 derivatives:
 \begin{equation}
 \big({\nabla}_{\!jk}-{\nabla}_{\!kj}
 \big)u_{i_1\dots i_m} =
 \sum_{a=1}^mR^p_{i_akj}
 u_{i_1\dots i_{a-1}pi_{a+1}\dots i_m},
 \label{(3.14)}
 \end{equation}
  where $\big(R^p_{ijk}\big)$~is the curvature tensor.

 \begin{Lemma}\label{3.6}
 For $v\in C^\infty(S^m\tau')$ and $p\ge 0$$,$
 \begin{align*}
 \nabla_{\!j_1\dots j_{m+p}}v_{i_1\dots i_m}
 &=\sigma(i_1\dots i_m)\sigma(j_1\dots j_{m+p})
 \sum_{l=0}^m(-1)^l{p+l-1\choose l}{m+p\choose m-l}\\
 &\quad\times
 \nabla_{\!i_{m-l+1}\dots i_mj_{l+p+1}\dots j_{m+p}}
 (d^pv)_{i_1\dots i_{m-l}j_1\dots j_{l+p}}
 \pmod{\nabla^{m+p-2}v}.
 \end{align*}
  \end{Lemma}

 \begin{proof} Define the tensors~$u$ and~$f$
 as follows:
  \begin{equation}
 \aligned
 u_{i_1\dots i_mj_1\dots j_pk_1\dots k_m}
 &=\sigma(j_1\dots j_pk_1\dots k_m)
  \nabla_{j_1\dots j_pk_1\dots k_m}v_{i_1\dots i_m},\\
 f_{i_1\dots i_mj_1\dots j_pk_1\dots k_m}
 &=\sigma(i_1\dots i_mj_1\dots j_p)\sigma(k_1\dots k_m)
 \nabla_{\!j_1\dots j_pk_1\dots k_m}v_{i_1\dots i_m}.
 \endaligned
 \label{(3.15)}
 \end{equation}
 By Lemma \ref{3.5}, we have
  \begin{equation}
 \nabla_{\!j_1\dots j_pk_1\dots k_m}
 v_{i_1\dots i_m}= u_{i_1\dots i_mj_1\dots j_pk_1\dots k_m}
 \pmod{\nabla^{m+p-2}v}.
 \label{(3.16)}
 \end{equation}
 Applying the operator $\sigma(i_1\dots i_mj_1\dots j_p)\sigma(k_1\dots k_m)$ to this equation and using~(\ref{(3.15)}),
 we obtain
 $$
 \sigma(i_1\dots i_mj_1\dots j_p)
 u_{i_1\dots i_mj_1\dots j_pk_1\dots k_m} =
 f_{i_1\dots i_mj_1\dots j_pk_1\dots k_m}
 \pmod{\nabla^{m+p-2}v}.
 $$
 Hence, $u$ and~$f$ satisfy the hypotheses of Lemma~2.1. Application of the lemma gives
 \begin{align}
 u_{i_1\dots i_mj_{1}\dots j_{m+p}}
 &=\sigma(i_1\dots i_m)\sigma(j_1\dots j_{m+p})
   \sum_{l=0}^m(-1)^l{p+l-1\choose m-l}{m+p\choose m-l}\nonumber\\
 &\quad\times
 f_{i_1\dots i_{m-l}j_1\dots j_{m+p}i_{m-l+1}\dots i_m}
 \pmod{\nabla^{m+p-2}v}.
 \label{(3.17)}
 \end{align}
  From~\eqref{(3.15)} and Lemma~\ref{3.5}, we deduce
 \begin{equation*}
 f_{i_1\dots i_mj_1\dots j_pk_1\dots k_m}
 =\sigma(k_1\dots k_m)
 \big(\nabla_{\!k_1\dots k_m}(d^pv)_{i_1\dots i_mj_1\dots j_p}
 \big)\pmod{\nabla^{m+p-2}v}.
 \end{equation*}
 Substituting the last expression into~\eqref{(3.17)}
 and using equality~\eqref{(3.16)}, we obtain the statement
 of the lemma.
 \end{proof}

 \begin{Lemma}\label{3.7}
 Let $u_i\in C^\infty(\otimes^{m_i}\tau')$ for
 $1\le i\le k$. If$,$ for every~$i$$,$ there exists $p_i\ge 1$
 such that
 $$
 \nabla^{p_i}u_i=0
 \pmod{\nabla^{p_1-1}u_1,\dots,\nabla^{p_k-1}u_k}
 $$
 and
 $
 u_i(x_0)=0,\,\, \nabla u_i(x_0)=0,\,\, \dots,\,\, \nabla^{p_i-1}u_i(x_0)=0
 $
 at some point $x_0$ of a connected manifold~$M$ then
 all the fields $u_i$ are identically equal to zero.
 \end{Lemma}

 \begin{proof}
We present only the scheme of the proof without details.
 In order to show that~$u_i$ vanish at some point~$x_1$, we connect the
 points~$x_0$ and~$x_1$ by a smooth curve~$x(t)$. The components of
 tensors $u_i(t)=u_i\big(x(t)\big)$ satisfy a linear homogeneous
 system of ordinary differential equations with homogeneous
 initial conditions. The order of the system with respect to~$u_i(t)$
 equals~$p_i$ and the system is solved with respect to
 the highest order derivatives. This implies the statement of the
 lemma.
 \end{proof}

The last two lemmas imply the following proposition.

 \begin{Lemma}\label{3.8}
 Assume $M$ to be connected and
 $u_i\in C^\infty(S^{m_i}\tau')$$,$ $1\le i\le\nobreak k$.
 If$,$ for every~$i$$,$ there exists~$p_i$ such that
 $$
 d^{\,p_i}u_i=0
 \pmod{\nabla^{p_1-1}u_1,\dots,\nabla^{p_k-1}u_k}
 $$
 and
 $$
 u_i(x_0)=0,\quad {\nabla}u_i(x_0)=0,\quad \dots,\quad
 {\nabla}^{m_i+p_i-1}u_i(x_0)=0
 $$
 at some point~$x_0$
 then all $u_i$ are identically equal to zero.
 \end{Lemma}

\section{Commutation formula for~$d$ and~$\delta$.
 The operator~$\Delta$}                     

 Let the operator
 $\Delta:C^\infty(S^m\tau')\rightarrow C^\infty(S^m\tau')$
 be defined as follows:
 \begin{equation}
 (\Delta u)_{i_1\dots i_m}=g^{jk}\nabla_{jk}u_{i_1\dots i_m}.
 \label{(6.1)}
 \end{equation}
   This differential operator has the order~2 and the degree~0,
  and acts on sections of the fiber bundle~$S^*\tau'\hskip-3pt .$
  Probably, (\ref{(3.15)}) is not the best definition of the
  Laplacian, and some zero order terms should be added to the right-hand side like for the Laplacian on differential forms.
  However, the most of our statements concerning~$\Delta$
  are formulated modulo low order terms,
  and such statements are independent of low
  order terms on the right-hand side of~\eqref{(6.1)}.

 \begin{Lemma} \label{6.1}
 The operator $\Delta$ is formally self-adjoint and satisfies
 the relations
  \begin{gather}
 \Delta i=i\Delta,
 \label{(6.2)}\\
 \Delta j=j\Delta,
 \label{(6.3)}\\
 \Delta^ld^ku=d^k\Delta^lu
 \pmod{\nabla^{k+2l-2}u},
 \label{(6.4)}\\
 \Delta^l\delta^ku=\delta^k\Delta^lu
 \pmod{\nabla^{k+2l-2}u}.
 \label{(6.5)}
 \end{gather}
 \end{Lemma}

 \begin{proof}
 As follows from  Green's formula,
 for $u,v\in C^\infty_0(S^*\tau')$,
 $$
 (\Delta u,v)_{L^2}=-(\nabla u,\nabla v)_{L^2}.
 $$
  This implies the first statement of the lemma since the
  right-hand side of the last formula is symmetric in~$u$ and~$v$.
  Equality~\eqref{(6.2)} is proved by a direct calculation
  in coordinates which is omitted, and \eqref{(6.3)} follows from~\eqref{(6.2)} since these relations are dual to each other.

 We now prove~\eqref{(6.4)} in the case of $k=l=1$.
 For $u\in C^\infty(S^m\tau')$, we have
  \begin{align*}
 (\Delta du)_{i_1\dots i_{m+1}}= \sigma(i_1\dots i_{m+1})(g^{jk}
 \nabla_{jki_{m+1}}u_{i_1\dots i_m}),\\
 (d\Delta u)_{i_1\dots i_{m+1}}= \sigma(i_1\dots i_{m+1})
 (g^{jk}\nabla_{i_{m+1}jk}u_{i_1\dots i_m}).
 \end{align*}
 By Lemma 7.1, the right-hand sides of these equalities coincide modulo ${\nabla}u$. This proves (8.4) for $k=l=1$. In the general case, (8.4) is proved by induction in $k$ and $l$.
Formula (8.5) follows from (8.4) by conjugation.
 \end{proof}

 We define the operator $R\in\Hom(S^*\tau',S^*\tau')$  by setting
 $$
 (Ru)_{i_1\dots i_m}= \sum_{a=1}^mg^{ij}R_{ii_a}u_{ji_1\dots\widehat{i}_a\dots i_m}+2
 \sum_{1\le a<b\le m}g^{ip}g^{jq}R_{ii_aji_b}u_{pqi_1\dots\widehat{i}_a
 \dots\widehat{i}_b\dots i_m}
 $$
 for $u\in S^m\tau'$. Here $(R_{ijkl})$~is the curvature tensor,
 and $\big(R_{ij}=g^{kl}R_{kijl}\big)$~is the Ricci tensor.
 The second sum on the right-hand side
 is absent in the case of~$m=1$, .

 \begin{Lemma}\label{6.2}
 The following commutation formula holds
 on $C^\infty(S^m\tau')$\dv
 $$
 \delta d=\frac1{m+1}(md\delta+\Delta-R).
 $$
 \end{Lemma}

 \begin{proof}
For $u\in C^\infty(S^m\tau')$$,$ we have
 \begin{align*}
 (m+1)(\delta du)_{i_1\dots i_m}
 &=(m+1)g^{i_{m+1}i_{m+2}}\nabla_{i_{m+2}}(du)_{i_1\dots i_{m+1}}\\
 &=g^{i_{m+1}i_{m+2}}
 \Bigg(\nabla_{i_{m+2}i_{m+1}}u_{i_1\dots i_m}+
      \sum_{a=1}^m\nabla_{i_{m+2}i_a}u_{i_1\dots\widehat{i_a}\dots i_{m+1}}
 \Bigg).
  \end{align*}
  This can be written in the form
  \begin{align*}
  (m+1)(\delta du)_{i_1\dots i_m}
 &=(\Delta u)_{i_1\dots i_m}+\sum_{a=1}^m{\nabla}_{i_a}
 \big(g^{i_{m+1}i_{m+2}}{\nabla}_{i_{m+2}}u_{i_1
 \dots\widehat{i_a}\dots i_{m+1}}
 \big)\\
 &\qquad-\sum_{a=1}^mg^{i_{m+1}i_{m+2}}
         \big(\nabla_{i_ai_{m+2}}-\nabla_{i_{m+2}i_a}
         \big)u_{i_1\dots\widehat{i_a}\dots i_{m+1}}.
 \end{align*}
  Denote the last sum on the right-hand side of this equality by $A_{i_1\dots i_m}$ and rewrite the formula as
   \begin{equation}
 (m+1)(\delta du)_{i_1\dots i_m} =(\Delta u)_{i_1\dots i_{m}}+
 m(d\delta u)_{i_1\dots i_{m}}-A_{i_1\dots i_m}.
 \label{(6.6)}
 \end{equation}
 According to~\eqref{(3.14)}, we have
 \vskip-0.7pt\noindent
 $$
 A_{i_1\dots i_m}= g^{i_{m+1}i_{m+2}}
 \sum_{a=1}^m
 \sum_{\stackrel{b=1}{b\ne a}}^{m+1}R^p_{i_bi_{m+2}i_a}
 u_{i_1\dots\widehat{i}_a\dots i_{b-1}pi_{b+1}\dots i_{m+1}}.
 $$
We distinguish the summands corresponding to $b=m+1$. Then
 \begin{align*}
 A_{i_1\dots i_m}
 &=\sum_{a=1}^mg^{i_{m+1}i_{m+2}}
 R^p_{i_{m+1}i_{m+2}i_a}u_{i_1\dots
 \widehat{i}_a\dots i_mp}\\
 &\qquad+2
 \sum_{1\le a<b\le m}g^{i_{m+1}i_{m+2}}R^p_{i_bi_{m+2}i_a}
 u_{i_1\dots\widehat{i}_a\dots\widehat{i}_b\dots i_{m+1}p}
 =(Ru)_{i_1\dots i_m}.
 \end{align*}
 The statement of the lemma immediately follows by
 substituting the last expression into~\eqref{(6.6)}.
 \end{proof}

  We are going to use only the following corollary of
  Lemma~\ref{6.2}:
 \begin{equation}
 \delta du= \frac{m}{m+1}d\delta u+
 \frac1{m+1}\Delta u
 \pmod{u}\ \ \mbox{for}\ \ u\in C^\infty(S^m\tau').
 \label{(6.7)}
 \end{equation}

 \begin{Lemma}\label{6.3}
 For arbitrary nonnegative integers $m,$ $k,$ and~$l$$,$ the equality
  \begin{equation}
 \delta^ld^ku= \frac{l!(m+k-l)!}{(m+k)!}
 \sum_p{k\choose p}{m\choose l-p}d^{k-p}\Delta^p\delta^{l-p}u
 \pmod{\nabla^{k+l-2}u}
 \label{(6.8)}
 \end{equation}
 holds for all $u\in C^\infty(S^m\tau')$. The summation
 in~\eqref{(6.8)} is taken over all integers~$p$ under the agreement$:$
  \begin{equation}
  {i\choose j}=0
  \ \ \mbox{for $j<0$\, or\, $i<j$},\,\,\,\mbox{and}\,\,\,
 i!=0
  \ \ \mbox{for $i<0$}.
  \label{(6.9)}
 \end{equation}
 \end{Lemma}

 \begin{proof}
Equality~\eqref{(6.8)} is trivial for $k=0$
 or $l=0$. For $k=l=1$, it coincides with~\eqref{(6.7)}.
 In the case of $l=1$ and and of an arbitrary~$k$,~\eqref{(6.8)} looks as follows:
 \begin{equation}
 \delta d^ku =\frac{m}{m+k}d^k\delta u+
 \frac{k}{m+k}d^{k-1}\Delta u
 \pmod{\nabla^{k-1}u}.
 \label{(6.10)}
 \end{equation}
  The trivial case $m=k=0$ is not considered here.
 We prove~\eqref{(6.10)} by induction in~$k$.
 Assume~\eqref{(6.10)} to be valid for some $k\ge 1$.
 Then
 \begin{align*}
 \delta d^{k+1}u
 &=\delta d^k(du)= \bigg(\frac{m+1}{m+k+1}d^k\delta+
       \frac{k}{m+k+1}d^{k-1}
      \Delta\pmod{\nabla^{k-1}}
 \bigg)\,du\\
 &=\frac{m+1}{m+k+1}d^k\delta\,du+
  \frac{k}{m+k+1}d^{k-1}\Delta\,du
 \pmod{\nabla^{k}u}.
 \end{align*}
Taking~\eqref{(6.4)} and~\eqref{(6.7)} into account, this gives
  \begin{align*}
 \delta d^{k+1}u
 &=\frac{m+1}{m+k+1}d^k
 \bigg(\frac{m}{m+1}\,d\delta u+
       \frac{1}{m+1}\,\Delta\, u\pmod{u}
 \bigg)\\
 &\qquad+\frac{k}{m+k+1}d^k\Delta\,u
 \pmod{\nabla^ku}\\
 &=\frac{m}{m+k+1}d^{k+1}\,\delta u+
   \frac{k+1}{m+k+1}d^k\,\Delta\, u
 \pmod{\nabla^ku}.
 \end{align*}
  The last relation coincides with~\eqref{(6.10)} for $k:=k+1$.
 Hence,~\eqref{(6.10)} is proved.

 Equality~\eqref{(6.8)} is trivial for $l>k+m$ since, in this case,
 both its sides are equal to zero. Hence it suffices to prove~\eqref{(6.8)}
 for $ 1\le l\le k+m $.
 We use induction in~$l$. For $l=1$, equality~\eqref{(6.8)} is already
 established. Assume now that~\eqref{(6.8)} is satisfied
 for some $1\le l<k+m$. Then
 $$
 \delta^{l+1}d^ku
  =\delta(\delta^ld^ku)
  =\frac{l!(m+k-l)!}{(m+k)!}
 \sum_p{k\choose p}{m\choose l-p}
 (\delta d^{k-p})\Delta^p\delta^{l-p}u
 \pmod{\nabla^{k+l-1}u}.
 $$
  Using~\eqref{(6.10)} and~\eqref{(6.5)}, we transform the last formula to the following one:
 \begin{align*}
 \delta^{l+1}d^ku
 &=\frac{l!(m+k-l)!}{(m+k)!}
   \sum_p{k\choose p}{m\choose l-p}\\
 &\times
 \bigg[\frac{m-l+p}{m-l+k}d^{k-p}\Delta^p\delta^{l-p+1}u
 +
       \frac{k-p}{m-l+k}d^{k-p-1}\Delta^{p+1}\delta^{l-p}u
  \bigg]\pmod{\nabla^{k+l-1}u}.
 \end{align*}
   Combining the first summand in the brackets of the
  $ p$th term and the second summand of the $(p-1)$th term, we arrive to~\eqref{(6.8)} for $l:=l+1$.
 \end{proof}

 \pagebreak
  \begin{Lemma}\label{6.4}
 For arbitrary nonnegative integers~$m$ and~$k$$,$ the equality
 \begin{align}
 \hskip-4mm\delta^kiu=\frac1{(m\+1)(m\+2)}
 &\Big(2k(m\_k\+2)d\delta^{k\_1}u\+k(k\_1)\Delta\delta^{k\_2}u\nonumber\\
 \noalign{\vskip-7pt}
       &\quad\+(m\_k\+1)(m\_k\+2)i\delta^ku
  \Big)\pmod{\nabla^{k\_2}u}
 \label{(6.11)}
 \end{align}
  is valid for all $u\in C^\infty(S^m\tau')$.
 If $m+k\ge 2$ then
  \begin{align}
 \hskip-5mm
 jd^ku=\frac1{(m+k-1)(m+k)}
 &\Big(2kmd^{k-1}\delta u+k(k-1)d^{k-2}\Delta u\nonumber\\
 \noalign{\vskip-7pt}
       &\qquad+m(m-1)d^kju
  \Big)\pmod{\nabla^{k-2}u}.
 \label{(6.12)}
 \end{align}
 \end{Lemma}

 \begin{proof}
These equalities are dual to each other.
Hence it suffices to prove the second one.

 We  prove~\eqref{(6.12)} by induction in~$k$.
 This equality is trivial for $k=0$ and coincides with~(3.6)
 for $k=1$. Assume~\eqref{(6.12)} to be valid for some $k\ge 1$.
 Then
  \begin{align*}
 jd^{k+1}u
 &=(jd^k)du
 =\frac1{(m+k)(m+k+1)}
 \Big(2k(m+1)d^{k-1}\delta du+k(k-1)d^{k-2}\Delta du\\
 \noalign{\vskip-7pt}
      &\kern60mm +m(m+1)d^kjdu
 \Big)\pmod{\nabla^{k-1}u}.
 \end{align*}
  \goodbreak

 Transforming each summand on the right-hand side
 according to Lemmas~\ref{6.3},~\ref{6.1}, and~3.3,
 respectively, we arrive to~\eqref{(6.12)}
 for $ k:=k+1$.
 \end{proof}

\section{Proof of theorem~1.1 in the case of
  $n=\dim\,M\ge3$}

  In the case of $m=0$, equation~(1.3) reduces to $ du=0 $ for a scalar
  function~$u$, and Theorem~1.1 is obvious in this case. Hence
  we assume $ m\ge 1$ in this section.

  Roughly speaking, the next lemma allows us to eliminate~$u$
  from equations~(1.2)--(1.3).

 \begin{Lemma}\label{7.1}
 Let $u\in C^\infty(S^m\tau')$ and
  $v\in C^\infty(S^{m-1}\tau')$ satisfy~{\rm(1.2)--(1.3)}.
 Then
 \begin{gather}
 (n+2m-4)d^{m+1}v=i
 \big((m-1)d^m\delta v-d^{m-1}\Delta v
 \big)\pmod{\nabla^mu},
 \label{(7.1)}\\
 jv=0,
 \label{(7.2)}\\
 v=0\pmod{\nabla u}.
 \label{(7.3)}
 \end{gather}
 \end{Lemma}

 \begin{proof}
 Applying the operator $j$ to equation~(1.3) and using Lemma~3.3, we obtain
 $$(m+1)jiv=(m+1)jdu=2\delta u+(m-1)dju.$$
 Since $ju=0$, this gives
  \begin{equation}
 v=\frac2{m+1}(ji)^{-1}\delta u.
 \label{(7.4)}
 \end{equation}
 Observe that $j(\delta u)=\delta ju=0$.
 Applying Lemma~2.2 to~$\delta u$, we obtain
 $$(ji)^{-1}\delta u=\frac{m(m+1)}{2(n+2m-2)}\delta u .$$
 Substitute this expression into \eqref{(7.4)} to obtain
 \begin{equation}
 v=\frac{m}{n+2m-2}\delta u.
 \label{(7.5)}
 \end{equation}
 In particular, this implies~\eqref{(7.2)} and~\eqref{(7.3)}.

 In order to prove~\eqref{(7.1)}, we introduce the temporary
 notation $f=iv$. Equation~(1.3) can be written as $du=f$.
The latter equation can be solved in $\nabla^{m+1}u$.
  Indeed, applying Lemma~\ref{3.6} with $p=1$,  we obtain
 \pagebreak
 \begin{align}
 \hskip-3mm\nabla_{j_1\dots j_{m+1}}u_{i_1\dots i_m}
 &=\sigma(i_1\dots i_m)\sigma(j_1\dots j_{m+1})
 \sum_{l=0}^m(-1)^l{m+1\choose l+1}\nonumber\\
 &\qquad\times
 {\nabla}_{i_{m-l+1}\dots i_mj_{l+2}\dots j_{m+1}}
 f_{i_1\dots i_{m-l}j_1\dots j_{l+1}}
 \pmod{\nabla^{m-1}u}.
 \label{(7.6)}
 \end{align}

 Our further arguments are different in the cases of
 $m=1$ and of $m>1$.
 We first assume~$m>1$.
 We contract equation~\eqref{(7.6)} with
  $g^{i_{m-1}i_m}$ (i.e., multiply this equation by
  $g^{i_{m-1}i_m}$ and take the sum over~$i_{m-1}$ and~$i_m$).
  In virtue of the equality
 $$
 g^{i_{m-1}i_m}\nabla_{j_1\dots j_{m+1}}u_{i_1\dots i_m}= \nabla_{j_1\dots j_{m+1}}(ju)_{i_1\dots i_{m-2}}=0,
 $$
 we obtain
 \begin{align}
 &\sigma(j_1\dots j_{m+1})
 \sum_{l=1}^{m+1}(-1)^l{m+1\choose l}
 g^{i_{m-1}i_m}\nonumber\\
 &\ \ \times\sigma(i_1\dots i_m)
 {\nabla}_{i_{m-l+2}\dots i_mj_{l+1}\dots j_{m+1}}
 f_{i_1\dots i_{m-l+1}j_1\dots j_{l}}=0\pmod{\nabla^{m-1}u}.
 \label{(7.7)}
 \end{align}

  Using Lemma~\ref{3.5} and the equality
  $ f=iv=0\pmod{\nabla u}$ which follows from~\eqref{(7.3)},
  we can permute the indices in each factor of the product
 $$
 \nabla_{i_{m-l+2}\dots i_mj_{l+1}\dots j_{m+1}}
 f_{i_1\dots i_{m-l+1}j_1\dots j_l}
 $$
  without violating equation~\eqref{(7.7)}. Using this observation, we divide all summands of the sum in~\eqref{(7.7)} to three groups so that
  the indices~$i_{m-1}$
  and~$i_m$ belong to the first (second) factor in all summands
  of the first (third) group, and belong to different factors in
  the summands of the second group. Rename these indices as
   $ i_{m-1}=p_1$ and $i_m=p_2$ for clarity. In such the way we obtain
 \begin{align*}
 &\sigma(i_1\dots i_{m-2})\sigma(j_1\dots j_{m+1})
  \sum_{l=1}^{m+1}(-1)^l{m+1\choose l}g^{p_1p_2}\\
  \noalign{\vskip-7pt}
 &\times
 \Big((l-1)(l-2)\nabla_{i_{m-l+2}
       \dots i_{m-2}j_{l+1}
       \dots j_{m+1}p_1p_2}f_{i_1
       \dots i_{m-l+1}j_1
       \dots j_l}\\
      &+2(l-1)(m-l+1)\nabla_{i_{m-l+1}
       \dots i_{m-2}j_{l+1}
       \dots j_{m+1}p_2}f_{i_1
       \dots i_{m-l}j_1
       \dots j_lp_1}\\
      &+(m{-}l)(m{-}l{+}1)\nabla_{i_{m{-}l}
       \dots i_{m{-}2}j_{l{+}1}
       \dots j_{m{+}1}}f_{i_1
       \dots i_{m{-}l{-}1}j_1
       \dots j_lp_1p_2}
 \Big){=}\,0\!\!\pmod{\nabla^{m-1}u}.
 \end{align*}
   Now, we symmetrize this equation in all free indices
  (i.e., apply the operator\\
   $\sigma(i_1\dots i_{m-2}j_1\dots j_{m+1})$). Changing simultaneously notations as $j_1=i_{m-1},\dots,j_{m+1}=i_{2m-1}$, we get
 \begin{align*}
 &\sigma(i_1\dots i_{2m-1})
  \sum_{l=1}^{m+1}(-1)^l{m+1\choose l}g^{p_1p_2}\\
 \noalign{\vskip-3pt}
 &\qquad\times
 \Big[(l-1)(l-2)\nabla_{i_{m+2}
      \dots i_{2m-1}p_1p_2}f_{i_1
      \dots i_{m+1}}\\
      &\qquad\qquad+2(l-1)(m-l+1)\nabla_{i_{m+1}\dots i_{2m-1}p_2}
                   f_{i_1\dots i_{m}p_1}\\
      &\qquad\qquad+(m-l)(m-l+1){\nabla}_{i_m\dots i_{2m-1}}
                   f_{i_1\dots i_{m-1}p_1p_2}
 \Big]=0\pmod{\nabla^{m-1}u}.
  \end{align*}
Observe that, for different values of~$ l $, the values of the first (second, third) summand in the
brackets differ only by
some factors. So the equation is transformed to the following form:
 \begin{align}
 \sigma(i_1\dots i_{2m-1})
 &\Big(a{\nabla}_{i_{m+2}\dots i_{2m-1}p_1p_2}
         f_{i_1\dots i_{m+1}}+b{\nabla}_{i_{m+1}\dots i_{2m-1}p_2}
         f_{i_1\dots i_{m}p_1}\nonumber\\
 &\qquad+c{\nabla}_{i_m\dots i_{2m-1}}
         f_{i_1\dots i_{m-1}p_1p_2}
 \Big)g^{p_1p_2}=0
 \pmod{\nabla^{m-1}u},
 \label{(7.8)}
 \end{align}
 where
 \vskip-10mm
 \begin{align*}
 a&=\sum_{l=1}^{m+1}(-1)^l{m+1\choose l}(l-1)(l-2)=-2,\\
 b&=2\sum_{l=1}^{m+1}(-1)^l{m+1\choose l}(l-1)(m-l+1)=2(m+1),\\
 c&=\sum_{l=1}^{m+1}(-1)^l{m+1\choose l}(m-l)(m-l+1)=-m(m+1).
 \end{align*}
Substituting these values for the coefficients, we write~\eqref{(7.8)} in the
 coordinate-free form
 \begin{equation}
 2d^{m-2}\Delta f-2(m+1)d^{m-1}\delta f+m(m+1)d^mjf=0
 \pmod{\nabla^{m-1}u}.
 \label{(7.9)}
 \end{equation}

 We recall that $ f=iv $ and express all summands of~\eqref{(7.9)}
 in terms of~$ v$. Lemma~2.2 and~\eqref{(7.2)} imply
 \begin{equation}
 jf=jiv=\frac{2(n+2m-2)}{m(m+1)}v.
 \label{(7.10)}
 \end{equation}
With the help of Lemma~\ref{3.3}, we deduce
  \begin{equation}
 \delta f=\delta iv =\frac2{m+1}\,dv+
 \frac{m-1}{m+1}i\delta v.
 \label{(7.11)}
 \end{equation}
 Since the operators $i$ and~$\Delta$ commute, we have
  \begin{equation}
 \Delta f=\Delta iv=i\Delta v.
 \label{(7.12)}
 \end{equation}

Substituting \eqref{(7.10)}--\eqref{(7.12)}
 into~\eqref{(7.9)} and using the commutation formula $ di=id$, we arrive to the relation
  \begin{equation}
 (n+2m-4)d^mv=i
 \big((m-1)d^{m-1}\delta v-d^{m-2}\Delta v
 \big)
 \pmod{\nabla^{m-1}u}.
 \label{(7.13)}
 \end{equation}
 Applying the operator $ d $ to this equation, we obtain~\eqref{(7.1)}.

 Let us now consider the case of $m=1$. Equation~\eqref{(7.6)} takes the form
 $$
 \nabla_{j_2j_3}u_i =
 \nabla_{j_2}f_{ij_3}+
 \nabla_{j_3}f_{ij_2}-
 \nabla_{i}f_{j_2j_3}\pmod{u}.
 $$
 After differentiation we obtain
  \begin{equation}
 \nabla_{j_1j_2j_3}u_i =
 \nabla_{j_1j_2}f_{ij_3}+
 \nabla_{j_1j_3}f_{ij_2}-
 \nabla_{j_1i}f_{j_2j_3}
 \pmod{\nabla u}.
 \label{(7.14)}
 \end{equation}
  According to Lemma~\ref{7.1}, the third order derivatives
 satisfy the relation
 $$
 \nabla_{j_1j_2j_3}u_i-{\nabla}_{j_2j_1j_3}u_i=0
 \pmod{\nabla u}.
 $$
 Inserting~\eqref{(7.14)} into the last equation, we obtain
 $$
 \nabla_{j_1j_3}f_{ij_2}+\nabla_{ij_2}f_{j_1j_3}-
 \nabla_{j_2j_3}f_{ij_1}-\nabla_{ij_2}f_{j_2j_3}=0
 \pmod{\nabla u}.
 $$
 Now, substituting $ f=iv$, we deduce
 $$
 g_{ij_2}\nabla_{j_1j_3}v+g_{j_1j_3}\nabla_{ij_2}v-
 g_{ij_1}\nabla_{j_2j_3}v-g_{j_2j_3}\nabla_{ij_1}v=0
 \pmod{\nabla u}.
 $$
 Contracting this equation with $g^{j_1j_3} $, we arrive to the formula
 $$
 (n-2)\nabla_{ij}v=-(\Delta v)g_{ij}
 \pmod{\nabla u}
 $$
that coincides with~\eqref{(7.1)} for $m=1$.
 \end{proof}

 Observe that, in the case of $m>1 $, we have established
 relation~\eqref{(7.13)} which is stronger than that of Lemma~\eqref{(7.1)}.

 The next statement plays the main role in the proof of Theorem~1.1.

 \begin{Lemma}\label{7.2}
 Assume $ u\in C^\infty(S^m\tau')$ and $ v\in C^\infty(S^{m-1}\tau')$
 to satisfy {\rm(1.2)} and {\rm(1.3)} for $m\ge 1$.
 Then$,$ for every integer~$ l $ such that $ 0\le l\le 2m $$,$
 the equation
 \begin{equation}
 \sum_{p=p_1}^{p_2}a_pd^{m-l+p+1}\Delta^{l-p}\delta^pv=i
 \sum_{p=p_3}^{p_4}b_pd^{m-l+p-1}\Delta^{l-p+1}\delta^pv
 \pmod{\nabla^{m+l}u}
 \label{(7.15)}
 \end{equation}
  is valid with some rational coefficients $ a_p=a_p(n,m,l) $
 and $ b_p=b_p(m,l) $. Here the summation limits are defined as
 follows$:$
 \begin{equation}
 \aligned
 p_1&=p_1(m,l)=\max(0,l-m-1),\\
 p_2&=p_2(m,l)=\min(m-1,l),\\
 p_3&=p_3(m,l)=\max(0,l-m+1),\\
 p_4&=p_4(m,l)=\min(m-1,l+1).
 \endaligned
 \label{(7.16)}
 \end{equation}
  The coefficients $a_{p_1}$ and $a_{p_2}$ are not equal to zero.
 \end{Lemma}

 \begin{proof}
 Apply the operator $\delta^l$ to equation~\eqref{(7.1)}
 $$
 (n+2m-4)\delta^ld^{m+1}v= \delta^li\big((m-1)d^m\delta v-d^{m-1}\Delta v
          \big)
 \pmod{\nabla^{m+l}u}.
 $$
 We transform the right-hand side of this equality with the help of Lemma~\ref{6.4} and obtain
  \begin{align*}
 &2m(2m-1)(n+2m-4)\delta^ld^{m+1}v-2l(2m-l)(m-1)d\delta^{l-1}d^m\delta v\\\noalign{\vskip4pt}
 &\quad-l(l\_1)(m\_1)\Delta\delta^{l\_2}d^m\delta v\+
          2l(2m\_l)d\delta^{l\_1}d^{m\_1}\Delta v\+
 l(l\_1)\Delta \delta^{l\_2}d^{m\_1}\delta v\\\noalign{\vskip4pt}
 &\qquad=(2m-l)(2m-l-1)i
 \big[(m-1)\delta^ld^m\delta v-\delta^ld^{m-1}\Delta v
 \big]\pmod{\nabla^{m+l}u}.
 \end{align*}
 Taking~\eqref{(7.3)} into account, we transform each summand on the left-hand side
 and the summands in the brackets to the form
  $ d^r\Delta^s\delta^tv $ by using the commutation formulas for powers
 of~$d$, $\delta$, and~$\Delta$ (see Lemmas~\ref{6.1}
 and~\ref{6.3}). Elementary but cumbersome calculations lead
  us to equation~\eqref{(7.15)} where the summation is taken
  over all integers~$ p $ under the agreement $ d^k=\delta^k=\Delta^k=0 $ for $ k<0 $,
  and the coefficients are as follows:
 \begin{align}
 a_p&=\left[(n+2m-4){m+1\choose l-p}+2
                    {m-1\choose l-p-1}+
                    {m-1\choose l-p-2}
      \right]{m-1\choose p}\nonumber\\\noalign{\vskip5pt}
 &\qquad-(m-1)
 \left[2{m\choose l-p}+
        {m\choose l-p-1}
 \right]{m-2\choose p-1},
 \label{(7.17)}\\\noalign{\vskip14pt}
 b_p&=(m-1){m\choose l-p+1}
           {m-2\choose p-1}-
           {m-1\choose l-p}
           {m-1\choose p}.
 \label{(7.18)}
 \end{align}
Agreement~\eqref{(6.9)} is used in \eqref{(7.17)} and \eqref{(7.18)}.

  Elementary arithmetical analysis of formula~\eqref{(7.17)} shows that the
  coefficients~$ a_p $ can be nonzero for $ p_1\le p\le p_2 $ only,
  where~$p_1$ and~$ p_2$ are defined in~\eqref{(7.16)},
and~$a_{p_1}$ and~$a_{p_2}$ are definitely nonzero. Similarly,
 ~\eqref{(7.18)} implies that $ b_p $ can be nonzero for $ p_3\le p\le p_4 $ only.
 \end{proof}
 \smallskip

\begin{proof}[\it Proof of Theorem~$1.1$]
  Recall that we assume $m\ge\nobreak 1$.
  First we prove by induction in~$ k $ the equality
  \begin{equation}
  \Delta^{m+k}\delta^{m-k}v=0
  \pmod{\nabla^{3m+k-1}u}
  \ \ \mbox{for}\ \ 0\le k\le m.
  \label{(7.19)}
  \end{equation}
    The equality is trivial for $k=0$ since $\delta^m v=0$.
  Assume~\eqref{(7.19)} to be valid for $k=0,\dots,s-1<m$.
 We write~\eqref{(7.15)} for $ l=2m-s+1 $ as follows:
 $$
 \sum_{p=m-s}^{m-1}a_pd^{s+p-m}\Delta^{2m-s+p+1}\delta^pv
 =i\sum_{p=m-s+2}^{m-1}b_pd^{s+p-m-2}\Delta^{2m-s-p+2}
 \delta^pv\pmod{\nabla^{3m-s+1}u}.
 $$
 We apply the operator $\Delta^{s-1}$
 to this equality and  transform all terms of the resulting formula to the form $ d^r\Delta^s\delta^tv $ with the help of \eqref{(6.4)}. By Lemma~\ref{7.2},~$a_{m-s} \neq 0 $.
 We distinguish the first summand on the left-hand side and write the result as follows:
 \begin{equation}
 \begin{aligned}
  \Delta^{m+s}\delta^{m-s}v&+\sum_{p=m-s+1}^{m-1}
 a'_pd^{s+p-m}\Delta^{2m-p}\delta^pv\\
 &=i\sum_{p=m-s+2}^{m-1}b'_pd^{s+p-m-2}\Delta^{2m-p+1}\delta^pv
 \pmod{\nabla^{3m+s-1}u}.
 \end{aligned}
 \label{(7.20)}
 \end{equation}

 By the inductive hypothesis,
 $$
 \Delta^{m+k}\delta^{m-k}v =0
 \pmod{\nabla^{3m+k-1}u}
 \ \ \mbox{for}\ \ 0\le k\le s-1.
 $$
 Setting $ k=m-p $ here, we have
 $$
 \Delta^{2m-p}\delta^{p}v=0
 \pmod{\nabla^{4m-p-1}u}
 \ \ \mbox{for}\ \ m-s+1\le p\le m-1.
 $$
Applying the operators
  $d^{s+p-m}$ and $d^{s+p-m-2}\Delta $ to this equation, we obtain
  \begin{alignat*}2
 d^{s+p-m}\Delta^{2m-p}\delta^{p}v=0
      &\pmod{\nabla^{3m+s-1}u}&&\ \ \mbox{for}\ \ m-s+1\le p\le m-1,\\
 d^{s+p-m-2}\Delta^{2m-p+1}\delta^{p}v=0
      &\pmod{\nabla^{3m+s-1}u}&&\ \ \mbox{for}\ \ m-s+2\le p\leq m-1.
 \end{alignat*}
 Both sums on~\eqref{(7.20)}
 are equal to zero $\pmod{\nabla^{3m+s-1}u}$, as follows from the last two equations. Hence,
 $$
 \Delta^{m+s}\delta^{m-s}v=0
 \pmod{\nabla^{3m+s-1}u}.
 $$
 This coincides with~\eqref{(7.19)} in the case of $k=s$.
 This completes the inductive step. Thus,~\eqref{(7.19)}
 is proved.

 Now we prove the equality
 \begin{equation}
 d^{m+2r+k-1}\Delta^{2m-r-k}\delta^{k}v=0
 \pmod{\nabla^{5m-2}u}
 \ \ \mbox{for}\  0\le r\le 2m,\  0\le k\le 2m-r
 \label{(7.21)}
 \end{equation}
  by double induction in $ r $ and~$k$.
 \goodbreak
 Setting  $k:=m-k$ in~\eqref{(7.19)}, we have
 $$
 \Delta^{2m-k}\delta^{k}v=0
 \pmod{\nabla^{4m-k-1}u}
 \ \ \mbox{for}\ \ 0\le k\le m.
 $$
 Applying the operator $d^{m+k-1} $ to
  this equality, we obtain~\eqref{(7.21)} for $r=0$.

 Assume now~\eqref{(7.21)} to be valid valid for $ 0\le r\le s-1<2m$,
 i.e.,
 \begin{equation}
 d^{m+2r+k-1}\Delta^{2m-r-k}\delta^{k}v=0
 \pmod{\nabla^{5m-2}u}
 \ \ \mbox{for}\  0\le r\le s-1,\  0\le k\le 2m-r.
 \label{(7.22)}
 \end{equation}
 We are going to prove~\eqref{(7.21)} for $r=s$. To this end we write down~\eqref{(7.15)} for $ l=0 $
 \big(this is exactly~\eqref{(7.1)}\big)
 as follows:
 $$
 (n+2m-4)d^{m+1}v=i
 \big((m-1)d^m\delta v-d^{m-1}\Delta v
 \big)\pmod{\nabla^{m}u}.
 $$
Applying the operator $d^{2s-2}\Delta^{2m-s}$ to
 this equality, we obtain
  \begin{align*}
 &(n+2m-4)d^{m+2s-1}\Delta^{2m-s}v\\
 &\qquad=i\big((m-1)d^{m+2s-2}\Delta^{2m-s}\delta v-d^{m-2s-3}\Delta^{2m-s+1}v
         \big)\pmod{\nabla^{5m-2}u}.
 \end{align*}
 Setting $ k=0 $ and $r=s-1$ in~\eqref{(7.22)}, and then setting $k=1$ and $r=s-1$ in~\eqref{(7.22)},
 we see that the right-hand side of the last formula
 equals zero $\pmod{\nabla^{5m-2}u}$. We have thus
 proved~\eqref{(7.21)} for $ r=s $ and $ k=0$.

 Assume now~\eqref{(7.21)} to be valid for $r=s$ and $0\le k\le t-1<2m-s$, i.e.,
  \begin{equation}
 d^{m+2s+k-1}\Delta^{2m-s-k}\delta^{k}v=0
 \pmod{\nabla^{5m-2}u}
 \ \ \mbox{for}\ \ 0\le k\le t-1.
 \label{(7.23)}
 \end{equation}
 We are going to prove~\eqref{(7.21)} for $ k=t $.

 If $t\ge m$, then~\eqref{(7.21)} is obviously true for $ k=t $, since  $\delta^tv=0$ in this  case. Therefore we assume
  \begin{equation}
 t\le\min(2m-s,m-1)
 \label{(7.24)}
 \end{equation}
  Recall also that
  \begin{equation}
 1\le s\le 2m.
 \label{(7.25)}
 \end{equation}

 Let us write down~\eqref{(7.15)} for $ l=t$.
 By~\eqref{(7.24)} and \eqref{(7.25)},
 this equation takes the form
 \begin{equation}
 \sum_{p=0}^ta_pd^{m-t+p+1}\Delta^{t-p}\delta^pv
 =i\sum_{p=0}^{p_4}b_pd^{m-t+p-1}
 \Delta^{t-p+1}\delta^p v
 \pmod{\nabla^{m+t}u},
 \label{(7.26)}
 \end{equation}
  where
  \begin{equation}
 p_4=\min(m-1,t+1).
 \label{(7.27)}
 \end{equation}
   According to Lemma~\ref{7.2}, the coefficient~$ a_t$ in~\eqref{(7.26)} is not zero. We distinguish the last summand on the left-hand side of~\eqref{(7.26)} and apply the operator
 $d^{2s+t-2}\Delta^{2m-s-t} $ to this equation
  \begin{align}
 &d^{m+2s+t-1}\Delta^{2m-s-t}\delta^tv+
 \sum_{p=0}^{t-1}a'_pd^{m+2s+p-1}\Delta^{2m-s-p}\delta^pv\nonumber\\
 \noalign{\vskip-2mm}
 &\qquad=i\sum_{p=0}^{p_4}b'_pd^{m+2s+p-3}\Delta^{2m-s-p+1}\delta^pv
 \pmod{\nabla^{5m-2}u}.
 \label{(7.28)}
 \end{align}
 The sum on the left-hand side of~\eqref{(7.28)} equals zero
 $\pmod{\nabla^{5m-2}u}$ by the inductive hypothesis~\eqref{(7.23)}.  We shall
 prove the same for the right-hand side.

Setting  $ r=s-1 $ and $ k=p $ in~\eqref{(7.22)}, we obtain
 \begin{equation}
 d^{m+2s+p-3}\Delta^{2m-s-p+1}\delta^pv=0
 \pmod{\nabla^{5m-2}u}
 \ \ \mbox{for}\ \ 0\le p\le 2m-s+1.
 \label{(7.29)}
 \end{equation}
 Inequalities~\eqref{(7.24)} and~\eqref{(7.27)} imply
 $ p_4\le 2m-s+1 $.  Therefore all the summands
 on the right-hand side of~\eqref{(7.28)} are equal to zero
 $ \pmod{\nabla^{5m-2}u} $ according to~\eqref{(7.29)}. Hence,~\eqref{(7.28)} implies
 $$
 d^{m+2s+t-1}\Delta^{2m-s-t}\delta^tv=0
 \pmod{\nabla^{5m-2}u}.
 $$
We have thus proved~\eqref{(7.23)} for $ k=t $.
 This completes the inductive step in $ k $ and~$ r $.
 Thus,~\eqref{(7.21)} is proved.

 Setting $ r=2m $ and $ k=0 $ in~\eqref{(7.21)}, we obtain
 \begin{equation}
 d^{5m-1}v=0
 \pmod{\nabla^{5m-2}u}.
 \label{(7.30)}
 \end{equation}
 According to (1.3), we have $ du=0\pmod{v} $.
 Applying the operator~$d^{5m-2}$ to this equality, we deduce
  \begin{equation}
 d^{5m-1}u=0
 \pmod{\nabla^{5m-2}v}.
 \label{(7.31)}
 \end{equation}

 We write the initial conditions~(1.4) in the form
  \begin{equation}
 u(x_0)=0,\quad \nabla u(x_0)=0,\quad \dots,\quad \nabla^{6m-2}u(x_0)=0.
 \label{(7.32)}
 \end{equation}
  From~\eqref{(7.3)} and~\eqref{(7.32)},
  \begin{equation}
 v(x_0)=0,\quad \nabla v(x_0)=0,\quad \dots,\quad \nabla^{6m-3}v(x_0)=0.
 \label{(7.33)}
 \end{equation}

 We have thus proved that~$u$ and~$v $ satisfy equations
 \eqref{(7.30)} and \eqref{(7.31)} and the initial
 conditions~\eqref{(7.32)} and~\eqref{(7.33)}.
 Applying Lemma~\ref{3.8}, we obtain $ u\equiv v\equiv 0 $
 This finishes the proof of Theorem~1.1 in the case of $ n\ge 3 $.
 \end{proof}

 Recall that the highest order of derivatives in the initial
 conditions~(1.4) is denoted by $ l(m) $.
 We have shown in the proof that
 $$
 l(m)\le 6m-2\ \ \mbox{if}\ \ m>0.
 $$
 As was mentioned after the statement of Theorem~1.1, this
 estimate is not sharp. The exact value $ l(m)=2m $ was found
 in~\cite{[Wr]} for $ m=2 $, and in~\cite{[Ca],[Ea]} for an
 arbitrary~$ m $. In the same papers, the upper bound
 $$
 \frac{(n+m-3)!(n+m-2)!(n+2m-2)(n+2m-1)(n+2m)}
     {m!(m+1)!(n-2)!n!}
 $$
 was established for the dimension of the space of trace-free
 conformal Killing symmetric tensor
 fields of the rank~$ m $ on a manifold of dimension $ n\ge 3 $.
 Both the estimates are sharp and become equalities in the case of a conformal flat manifold.

  \section{Spherical harmonics Fourier series expansion \\
  of solution to the kinetic equation}

 Let $ M $~be a Riemannian manifold. Recall the operator
$$
\lambda:C^\infty(S^*\tau')\rightarrow C^\infty(\Omega)
$$
 was defined in Section 2. Let~$ H $ be the vector field on $ T=TM $
 which generates the geodesic flow. The field is expressed by~(1.9)
 in local coordinates. The field is tangent to the submanifold
 $\Omega\subset T $ at points of the latter submanifold. Hence the field can be considered
 as a differential operator
 $H:C^\infty(\Omega)\rightarrow C^\infty(\Omega)$ on the
 submanifold.

 \begin{Lemma}\label{8.1}
 The following equality holds on $C^\infty(S^*\tau')$\dv
 \begin{equation}
 \lambda d=H\lambda .
 \label{(8.1)}
 \end{equation}
 \end{Lemma}

 Since $ \lambda $ and $ H $ are restrictions to~$\Omega $
 of some operators defined on~$ T $, it suffices to prove the equality
 \begin{equation}
 \varkappa d=H\varkappa,
 \label{(8.2)}
 \end{equation}
 where $ H $ is considered as an operator on~$ T $, and the operator
 $\varkappa:C^\infty(S^*\tau') \rightarrow C^\infty(T) $ has been defined in
~\S\,2. The last equality can be easily checked by calculations in coordinates, and we omit the calculations.

 Now, assume  functions $U,F\in C^\infty(\Omega)$ to be
 linked by the {\it kinetic equation}
 \begin{equation}
 HU=F.
 \label{(8.3)}
 \end{equation}
From~\eqref{(8.3)}, we will deduce some equations that relate the Fourier
 series of the functions $ U $ and~$ F$.

By Lemma~2.5, the functions $ U $ and $ F $ can be
uniquely  represented by the series
 \vskip-7pt\noindent
 \begin{alignat}3
 U&=\sum_{m=0}^{\infty}\lambda u_m,
  &&\quad u_m\in C^\infty(S^m\tau'),
    &&\quad ju_m=0,
 \label{(8.4)}\\
 F&=\sum_{m=0}^{\infty}\lambda f_m,
  &&\quad f_m\in C^\infty(S^m\tau'),
    &&\quad jf_m=0.
 \label{(8.5)}
 \end{alignat}
 As well known~\cite{[Sob]},  the Fourier series of a
 sufficiently smooth function on a sphere can be termwise differentiated
 with respect to the coordinates of a point of the sphere.
 The same is true for the differentiation with respect to
 the coordinates of a point $x\in M$ which play the role of
 parameters in the
 series. Hence,~\eqref{(8.4)} implies
 \begin{equation}
 HU=\sum_{m=0}^{\infty}H\lambda u_m
 \label{(8.6)}
 \end{equation}
   and the series
  converges absolutely and uniformly on any compact subset of~$\Omega$.

 According to Lemma~\ref{8.1}, we have
 \begin{equation}
 H\lambda u_m=\lambda du_m.
 \label{(8.7)}
 \end{equation}
  The condition $ ju_m=0 $ is equivalent to $ pu_m=u_m $.
 The last equality and Lemma~3.4 imply
 $$
 du_m=dpu_m=pdu_m+\frac{m}{n+2m-2}i\delta pu_m =
 pdu_m+\frac{m}{n+2m-2}i\delta u_m.
 $$
 Apply the operator $\lambda$ to this equality and use Lemma~2.4 to obtain
 $$
 \lambda du_m=\lambda pdu_m+\frac{m}{n+2m-2}\lambda\delta u_m.
 $$
 Comparing this formula with~\eqref{(8.7)}, we see
 $$
 H\lambda u_m=\lambda pdu_m+\frac{m}{n+2m-2}\lambda\delta u_m.
 $$
 Substitute this expression into~\eqref{(8.6)} to obtain
  \begin{equation}
 HU=\sum_{m=0}^{\infty}\lambda
 \bigg(pdu_{m-1}+\frac{m+1}{n+2m}\delta u_{m+1}
 \bigg).
 \label{(8.8)}
 \end{equation}
    For convenience, we assume here $u_{-1}=0$. The expression in parentheses in~\eqref{(8.8)}
  belongs to the kernel of~$j$ since $\delta$
  and~$j$ commute. Hence, \eqref{(8.8)} is the
  Fourier series of the function $ HU=\nobreak F $ with respect to
  the spherical harmonics, i.e., \eqref{(8.8)} must coincide
  with~\eqref{(8.5)}. We have thus proved the following

 \begin{Theorem}\label{8.2}
 Let $U\in C^\infty(\Omega)$ be a solution to the kinetic equation $ HU=F $$,$
 and let~\eqref{(8.4)} and~\eqref{(8.5)} be the spherical Fourier
 series of
 $ U $ and~$ F $$,$ respectively. Then
 \begin{gather*}
 \delta u_1=nf_0,\\
 pdu_m+\frac{m+2}{n+2m+2}\delta u_{m+2}=f_{m+1}
 \ \ \mbox{for}\ \ m=0,1,2,\dots,
 \end{gather*}
 where $n=\dim M$.
 \end{Theorem}

 \section{Proof of THeorem~1.1 in the two-dimensional case}

 We assume here $n=\dim\,M=2$. As well known, an  isothermic coordinate system~$(x,y)$ exists in some
 neighborhood of every point of a two-dimensional Riemannian manifold. In such a coordinate system,
   the Riemannian metric has the form
 \begin{equation}
 ds^2=e^{2\mu(x,y)}(dx^2+dy^2).
 \label{(9.1)}
 \end{equation}
 It suffices to prove Theorem~1.1 under the assumption that
 such a coordinate system is defined on the whole of~$M$.

 Define the coordinate system $(x,y,\theta)$ on the
 three-dimensional manifold~$\Omega$ such that $\theta$~ is the angle
 between the unit vector $ \xi\in\Omega$ and the coordinate line
 $ y=\const $. The spherical
 harmonics series
 expansion of
 a function
 $U(x,y,\theta)\in C^\infty(\Omega) $ coincides with the
 Fourier series with respect to~$\theta$. Hence the next
 statement is a
 particular case of Lemma~2.5
 (cf. the remark after statement
 of the lemma).

 \begin{Lemma}\label{9.1}
 Let $u\in C^\infty(S^m\tau')$ and let
 $$
 (\lambda u)(x,y,\theta)= \frac12a_0(x,y)+\sum_{k=1}^m
 \big(a_k(x,y)\cos k\theta+b_k(x,y)\sin k\theta
 \big)
 $$
 be the Fourier series
 of the function  $\lambda u\in C^\infty(\Omega)$. Then
 $$
 (\lambda pu)(x,y,\theta)=a_m(x,y)\cos m\theta+b_m(x,y)\sin m\theta.
 $$
 \end{Lemma}

 Assume~$u$ and~$v$ to satisfy the hypotheses of Theorem~1.1.
 The assumption $ ju=0 $ is equivalent to $ pu=u $.  Then, according to
 Lemma~\ref{9.1}, there exist functions $ a,b\in C^\infty(M) $
 such that
  \begin{equation}
 (\lambda u)(x,y,\theta)= a(x,y)\cos m\theta+b(x,y)\sin m\theta.
 \label{(9.2)}
 \end{equation}
 The condition $ du=iv $ is equivalent to $ pdu=0 $.
 According to Lemma~\ref{9.1}, this means that the
 Fourier
 series
 of the function~$ \lambda du $ does not contain the harmonics
 of order $ m+\nobreak 1$.
 Due
 to Lemma~\ref{8.1}, we have $ \lambda du=H\lambda u $.
 Hence the coefficients at $ \cos(m+1)\theta $ and
 $\sin(m+1)\theta$ in the Fourier
 series of the function~$H\lambda u$ are equal to zero identically in~$(x,y)$.

 The operator~$H$ has the following form in the coordinates~$(x,y,\theta)$:
  \begin{equation}
 H=e^{-\mu}
 \bigg(
 \cos\theta\frac{\partial}{\partial x}+
 \sin\theta\frac{\partial}{\partial y}+(-\mu_x\sin\theta+\mu_y\cos\theta)
          \frac{\partial}{\partial\theta}
 \bigg).
 \label{(9.3)}
 \end{equation}
 This can be derived from~(1.9) and \eqref{(9.1)} by a direct calculation that is omitted.

 Now, we express $ H\lambda u $ in terms of~$a$ and~$b$ using
 \eqref{(9.2)} and~\eqref{(9.3)}, and then expand $ H\lambda u $ in the Fourier series in~$\theta$.
 Equating to zero the coefficients of the series at $\cos(m+1)\theta$ and $\sin(m+1)\theta$, we arrive to the following equations:
 \begin{equation}
 \begin{array}{l}
 a_x-b_y-m(\mu_xa-\mu_yb)=0,\\
 [0.3cm]
 a_y+b_x-m(\mu_ya+\mu_xb)=0.
 \end{array}
 \label{(9.4)}
 \end{equation}

 Introducing the notation
 $$
 z=x+iy,\quad w=a+ib,\quad
 \frac{\partial}{\partial\bar z}= \frac12
 \bigg(\frac{\partial}{\partial x}+i
      \frac{\partial}{\partial y}
 \bigg),
 $$
 we write~\eqref{(9.4)} in the complex form
 $$
 \frac{\partial}{\partial\bar z}(e^{-m\mu}w)=0.
 $$
 So,  $e^{-m\mu}w$ is a holomorphic function. According to the hypotheses
 of Theorem~1.1, this function vanishes together with all its derivatives
 at some point. Hence it is identically zero.
 Now,~\eqref{(9.2)} shows that $ \lambda u\equiv 0$. Hence,
 $u\equiv 0$ and Theorem~1.1. is proved.

\section*{Acknowledgments}

 The authors are indebted to R.~Graham and M.~Eastwood for
 discussions, and to W.~Lionheart who has done
 a series of
 comments to the manuscript of the paper.

\end{document}